\documentclass[12pt]{amsart}
\usepackage{amssymb,euscript,amsmath, mathrsfs,amscd}
\usepackage{mathabx}
\usepackage{hyperref}
\usepackage{caption}
\usepackage{subcaption}

\usepackage[dvips]{color}
\usepackage{enumitem}
\usepackage{fullpage}
\usepackage{tikz}
\usepackage{tikz-cd}
\usetikzlibrary{shapes.geometric}
\usepackage{todonotes}

\usepackage{changes}
\definechangesauthor[name=Federico, color=red]{FC}
\definechangesauthor[name=Fu, color=blue]{L}

\usetikzlibrary{patterns}
\usetikzlibrary{calc}
\usetikzlibrary{intersections}
\usetikzlibrary{arrows.meta}

    \tikzset{smalltext/.style={"\textup{\small #1}" description}}

\def\namedlabel#1#2{\begingroup
#2%
\def\@currentlabel{#2}
\phantomsection\label{#1}\endgroup
}

\newtheorem{thm}{Theorem}[section]

\newtheorem{prop}[thm]{Proposition}
\newtheorem{lem}[thm]{Lemma}

\newtheorem{cor}[thm]{Corollary}

\theoremstyle{definition}
\newtheorem{defn}[thm]{Definition}
\newtheorem{defnex}[thm]{Definition/Example}

\newtheorem{construct}[thm]{Construction}
\newtheorem{ques}[thm]{Question}
\newtheorem{prob}[thm]{Problem}
\newtheorem{ex}[thm]{Example}

\newtheorem{rem}[thm]{Remark}

\numberwithin{equation}{section}

\newcommand*{\defeq}{\mathrel{\vcenter{\baselineskip0.5ex \lineskiplimit0pt
\hbox{\scriptsize.}\hbox{\scriptsize.}}}%
=}

\makeatletter
\def\moverlay{\mathpalette\mov@rlay}
\def\mov@rlay#1#2{\leavevmode\vtop{%
\baselineskip\z@skip \lineskiplimit-\maxdimen
\ialign{\hfil$\m@th#1##$\hfil\cr#2\crcr}}}
\newcommand{\charfusion}[3][\mathord]{
#1{\ifx#1\mathop\vphantom{#2}\fi
\mathpalette\mov@rlay{#2\cr#3}
}
\ifx#1\mathop\expandafter\displaylimits\fi}
\makeatother

\makeatletter
\def\subsubsection{\@startsection{subsubsection}{3}%
\z@{.5\linespacing\@plus.7\linespacing}{-.5em}%
{\normalfont\bfseries}}
\makeatother

\def\mm{\mid\mid}

\def\Z{{\mathbb Z}}
\def\R{{\mathbb R}}
\def\fS{{\mathfrak S}}
\def\fG{{\mathfrak G}}
\def\ds{\displaystyle}

\newcommand{\bm}[1]{{\boldsymbol{#1}}}
\def\0{\bm{0}}
\def\a{\bm{a}}

\def\e{\bm{e}}
\def\f{\bm{f}}

\def\r{\bm{r}}

\def\v{\bm{v}}
\def\V{\bm{V}}
\def\u{\bm{u}}
\def\w{\bm{w}}
\def\x{\bm{x}}
\def\y{\bm{y}}
\def\z{\bm{z}}

\def\balpha{\bm{\alpha}}

\def\bgamma{\bm{\gamma}}
\def\bdelta{\bm{\delta}}

\newcommand\preceqdot{\mathrel{\ooalign{$\preceq$\cr
\hidewidth\raise0.225ex\hbox{$\cdot\mkern0.5mu$}\cr}}}

\definecolor{amber}{rgb}{1.0, 0.75, 0.0}

\def\cA{\mathcal A}
\def\cB{\mathcal B}

\def\cD{\mathcal D}
\def\cF{{\mathcal F}}
\def\cG{\mathcal G}
\def\cH{\mathcal H}

\def\cO{\mathcal O}

\def\cR{\mathcal R}
\def\cS{\mathcal S}
\def\cT{\mathcal T}

\def\cZ{\mathcal Z}

\def\ffan{\mathrm{FFan}}

\def\carr{\kappa}

\newcommand{\ee}{\textbf{e}}

\def\sO{\mathscr O}

\def\fS{ {\mathfrak{S}}}
\def\fG{{\mathfrak{G}}}
\def\Gl{\mathfrak{Gl}}

\def\KK{{\mathscr K}}

\def\FL{\mathcal{F}}

\def\ncone{\operatorname{ncone}}
\def\conv{\operatorname{Conv}}
\def\cone{\operatorname{Cone}}

\def\std{\mathrm{std}}

\newcommand{\1}{\mathbf{1}}

\def\Type{ {\operatorname{Type}}}

\def\ncone{\operatorname{ncone}}

\def\1{ {\bf{1}}}

\numberwithin{equation}{section}

\definecolor{darkgreen}{rgb}{0,0.7,0}
\definecolor{brown}{rgb}{0.7,0.4,0}

\newcommand\commentout[1]{}
\author{Federico Castillo}
\address{Federico Castillo, Departamento de Matem\'aticas, Pontificia Universidad Cat\'olica de Chile, Santiago, Chile.}
\email{federico.castillo@mat.uc.cl}

\author{Fu Liu}
\address{Fu Liu, Department of Mathematics, University of California, Davis, One Shields Avenue, Davis, CA 95616 USA.}
\email{fuliu@ucdavis.edu}


\begin{document}
\title{Symmetrizing polytopes and posets}

\begin{abstract}
  Motivated by the authors' work on permuto-associahedra, which can be considered as a symmetrization of the associahedron using the symmetric group, we introduce and study the $\fG$-symmetrization of an arbitrary polytope $P$ for any reflection group $\fG$. 
	We show that the combinatorics, and moreover, the normal fan of such a symmetrization can be recovered from its refined fundamental fan, a decorated poset describing how the normal fan of $P$ subdivides the fundamental chamber associated to the reflection group $\fG$.

  One important application of our results is providing a way to approach the realization problem of a $\fG$-symmetric poset $\cF$, that is, the problem of constructing a polytope whose face poset is $\cF$. 
  Instead of working with the original poset $\cF,$ we look at its dual poset $\cT$ (which is $\fG$-symmetric as well) and focus on a generating subposet $\cZ$ of $\cT$, and reduce the problem to realizing $\cZ$ as a refined fundamental fan.   
\end{abstract}

\maketitle

\section{Introduction}

The five platonic solids have been known for thousands of years.
Ancient Greeks were fascinated by their symmetries, even though the precise notion of symmetry was not defined until much later.
In modern terms we can define the symmetric group of a polytope $P$ as the group of linear transformations $T$ such that $T(P)=P$;
see \cite{computing_symmetry} for a general survey about computing symmetries.
Around a hundred years ago, Schil\"afi studied regular polytopes, higher dimensional analogues of platonic solids, and classified them.
Formally, a regular polytope is a polytope such that its symmetric group acts transitively on its set of flags of faces.
In this paper we construct polytopes satisfying a weaker symmetric property.

\begin{defn}
Let $\fG \subset \mathfrak{Gl}(V)$ be a group of linear transformations.
We define a \emph{$ \fG $-symmetric polytope} to be a polytope $ P\subset V $ such that $ Pg= P $ for all $ g \in \fG $.
Equivalently, a $ \fG $-symmetric polytope is the convex hull of finitely many $\fG$-orbits.
\end{defn}

Polytopes that are $\fG$-symmetric for particular groups $\fG$ have been studied in the literature.
One example are \emph{permutation polytopes} \cite{baumeister2007permutation} where $\fG$ is the symmetric group acting on the set of matrices by matrix multiplication; the quintessential example is the Birkhoff polytope.
Another example is $ \fG=\mathbb{Z}/2\Z$  acting by negation ($ x\mapsto -x $).
In this case a $ \mathbb{Z}/2\Z$-symmetric polytope is known as a \emph{centrally symmetric} polytope and this has been extensively studied in particular with respect to their f-vectors \cite{novik2019tale}.

In this paper we focus on a particular class of groups: \textbf{finite reflection groups}.
By definition, these are finite groups of linear transformations generated by reflections on a vector space.
Examples include the symmetric groups of the platonic solids, and more generally, of regular polytopes.
There is a specific family of examples we are interested in this paper: the symmetric group of the $ (d-1) $-regular simplex which is known as the type-A Coxeter group and it is isomorphic to the symmetric group $\fS_d$. 

In the case where all the vertices of a $ \fG $-symmetric polytope $P$ form a single $\fG$-orbit, $P$ is called a $\fG$\emph{-permutohedron} (some sources called this orbit polytopes, weight polytopes, or $\Phi$-permutohedra when they want to emphasize in the root system rather than the reflection group).
See Figure \ref{fig:pi3} for a three dimensional permutohedron.
This case has been widely studied, see for instance \cite{postnikov2009permutohedra}; in particular, the face poset is well understood.
The general case, with multiple orbits, seems unexplored.
Recent examples of $ \fS_d $-symmetric polytopes are the lineup polytopes of hypersimplices \cite{castillo2021effective}  and of product of simplices \cite{castillo2023lineup}.
In this paper we systematically construct $ \fG $-symmetric polytopes by a process we call \emph{symmetrization}.

\begin{defn}
	\label{defn:symmetrization}
	For any polytope $P\subseteq V$, we define its \emph{$\fG$-symmetrization} to be the polytope
	\begin{equation}	
		\fG(P)=\conv\{\x g~:~g\in\mathfrak{G}, \x\in P\}.
	\end{equation}	
\end{defn}

The combinatorics of the symmetrization are non-trivial.
In Figure \ref{fig:intro} we have symmetrizations of two different triangles.
The one on the left, Figure \ref{fig:works2}, can be obtained from a permutohedron, Figure \ref{fig:pi3}, by a truncation of each vertex.
The one on the right, Figure \ref{fig:fail} cannot be obtained by repeated truncations of faces as it is not a simple polytope.

\iffalse
See Figure \ref{fig:intro} where two triangles are $ \fS_3 $-symmetrized.
To obtain a well-behaved operation we restrict to a certain kind of polytopes.
We required the polytope to have two properties.
First of all the polytopes cannot be too ``big''; in Figure \ref{fig:intro} the triangles must not overlap each other.
And also we want them to be well-positioned so that all of their vertices are vertices of the symmetrization (in Figure \ref{fig:placed} there is an example of how this can fail).
These two properties are formalized in our definitions of \textit{appropriate} (Definition \ref{def:appropriate}) and \textit{placed} (Definition \ref{def:placed}).
\fi

\begin{figure}[ht]
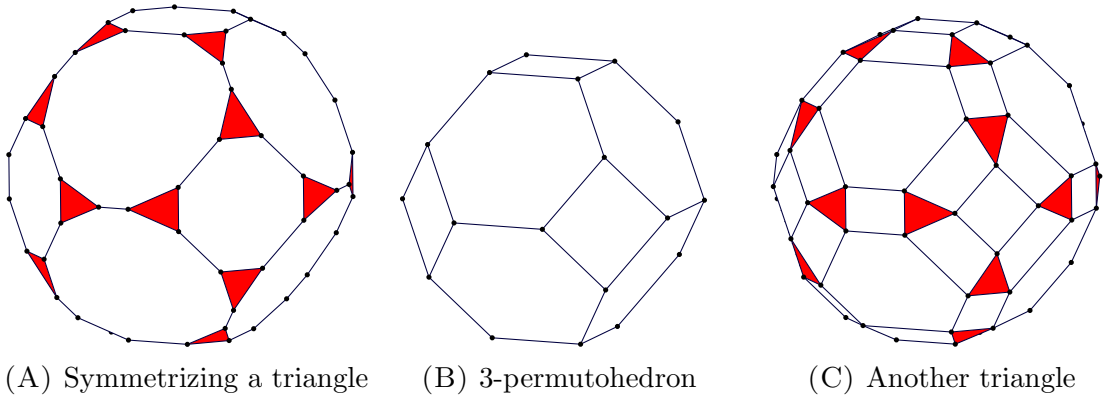

\begin{subfigure}{.3\textwidth}
\centering
\input{tikz/triangle_down.tex}
\caption{Symmetrizing a triangle}
\label{fig:works2}
\end{subfigure}%
\begin{subfigure}{.3\textwidth}
\centering
\begin{tikzpicture}%
	[x={(0.924815cm, 0.083972cm)},
	y={(-0.009943cm, 0.980336cm)},
	z={(0.380287cm, -0.178577cm)},
	scale=0.300000,
	back/.style={loosely dotted, thin},
	edge/.style={color=blue!25!black, thin},
	facet/.style={fill=white,fill opacity=0.800000},
	vertex/.style={inner sep=0.4pt,circle,draw=black,fill=black,thick}]
%
%
%% This TikZ-picture was produced with Sagemath version 10.3
%% with the command: ._tikz_3d_in_3d and parameters:
%% view = [100, 200, 25]
%% angle = 25
%% scale = 1
%% edge_color = blue!95!black
%% facet_color = blue!95!black
%% opacity = 0.8
%% vertex_color = green
%% axis = False
%%
%% Coordinate of the vertices:
%%
\coordinate (0, -6, 3) at (0, -6, 3);
\coordinate (0, 6, 3) at (0, 6, 3);
\coordinate (6, 0, 3) at (6, 0, 3);
\coordinate (6, -3, 0) at (6, -3, 0);
\coordinate (3, 6, 0) at (3, 6, 0);
\coordinate (6, 0, -3) at (6, 0, -3);
\coordinate (6, 3, 0) at (6, 3, 0);
\coordinate (3, -6, 0) at (3, -6, 0);
\coordinate (3, 0, 6) at (3, 0, 6);
\coordinate (3, 0, -6) at (3, 0, -6);
\coordinate (0, -3, 6) at (0, -3, 6);
\coordinate (0, 6, -3) at (0, 6, -3);
\coordinate (0, 3, 6) at (0, 3, 6);
\coordinate (-3, 6, 0) at (-3, 6, 0);
\coordinate (0, 3, -6) at (0, 3, -6);
\coordinate (-3, 0, 6) at (-3, 0, 6);
\coordinate (-6, 3, 0) at (-6, 3, 0);
\coordinate (-6, 0, 3) at (-6, 0, 3);
\coordinate (-6, -3, 0) at (-6, -3, 0);
\coordinate (-6, 0, -3) at (-6, 0, -3);
\coordinate (-3, 0, -6) at (-3, 0, -6);
\coordinate (0, -3, -6) at (0, -3, -6);
\coordinate (-3, -6, 0) at (-3, -6, 0);
\coordinate (0, -6, -3) at (0, -6, -3);
%%
%%
%% Drawing edges in the back
%%
\draw[edge,back] (6, -3, 0) -- (6, 0, -3);
\draw[edge,back] (6, 0, -3) -- (6, 3, 0);
\draw[edge,back] (6, 0, -3) -- (3, 0, -6);
\draw[edge,back] (3, -6, 0) -- (0, -6, -3);
\draw[edge,back] (3, 0, -6) -- (0, 3, -6);
\draw[edge,back] (3, 0, -6) -- (0, -3, -6);
\draw[edge,back] (0, 6, -3) -- (0, 3, -6);
\draw[edge,back] (0, 3, -6) -- (-3, 0, -6);
\draw[edge,back] (-6, 0, -3) -- (-3, 0, -6);
\draw[edge,back] (-3, 0, -6) -- (0, -3, -6);
\draw[edge,back] (0, -3, -6) -- (0, -6, -3);
\draw[edge,back] (-3, -6, 0) -- (0, -6, -3);
%%
%%
%% Drawing vertices in the back
%%
\node[vertex] at (6, 0, -3)     {};
\node[vertex] at (3, 0, -6)     {};
\node[vertex] at (0, 3, -6)     {};
\node[vertex] at (0, -3, -6)     {};
\node[vertex] at (0, -6, -3)     {};
\node[vertex] at (-3, 0, -6)     {};
%%
%%
%% Drawing the facets
%%
\fill[facet] (0, 3, 6) -- (0, 6, 3) -- (3, 6, 0) -- (6, 3, 0) -- (6, 0, 3) -- (3, 0, 6) -- cycle {};
\fill[facet] (0, -3, 6) -- (0, -6, 3) -- (3, -6, 0) -- (6, -3, 0) -- (6, 0, 3) -- (3, 0, 6) -- cycle {};
\fill[facet] (-3, 6, 0) -- (0, 6, 3) -- (3, 6, 0) -- (0, 6, -3) -- cycle {};
\fill[facet] (-3, 0, 6) -- (0, -3, 6) -- (3, 0, 6) -- (0, 3, 6) -- cycle {};
\fill[facet] (-6, 0, 3) -- (-3, 0, 6) -- (0, 3, 6) -- (0, 6, 3) -- (-3, 6, 0) -- (-6, 3, 0) -- cycle {};
\fill[facet] (-6, 0, -3) -- (-6, 3, 0) -- (-6, 0, 3) -- (-6, -3, 0) -- cycle {};
\fill[facet] (-3, -6, 0) -- (0, -6, 3) -- (0, -3, 6) -- (-3, 0, 6) -- (-6, 0, 3) -- (-6, -3, 0) -- cycle {};
%%
%%
%% Drawing edges in the front
%%
\draw[edge] (0, -6, 3) -- (3, -6, 0);
\draw[edge] (0, -6, 3) -- (0, -3, 6);
\draw[edge] (0, -6, 3) -- (-3, -6, 0);
\draw[edge] (0, 6, 3) -- (3, 6, 0);
\draw[edge] (0, 6, 3) -- (0, 3, 6);
\draw[edge] (0, 6, 3) -- (-3, 6, 0);
\draw[edge] (6, 0, 3) -- (6, -3, 0);
\draw[edge] (6, 0, 3) -- (6, 3, 0);
\draw[edge] (6, 0, 3) -- (3, 0, 6);
\draw[edge] (6, -3, 0) -- (3, -6, 0);
\draw[edge] (3, 6, 0) -- (6, 3, 0);
\draw[edge] (3, 6, 0) -- (0, 6, -3);
\draw[edge] (3, 0, 6) -- (0, -3, 6);
\draw[edge] (3, 0, 6) -- (0, 3, 6);
\draw[edge] (0, -3, 6) -- (-3, 0, 6);
\draw[edge] (0, 6, -3) -- (-3, 6, 0);
\draw[edge] (0, 3, 6) -- (-3, 0, 6);
\draw[edge] (-3, 6, 0) -- (-6, 3, 0);
\draw[edge] (-3, 0, 6) -- (-6, 0, 3);
\draw[edge] (-6, 3, 0) -- (-6, 0, 3);
\draw[edge] (-6, 3, 0) -- (-6, 0, -3);
\draw[edge] (-6, 0, 3) -- (-6, -3, 0);
\draw[edge] (-6, -3, 0) -- (-6, 0, -3);
\draw[edge] (-6, -3, 0) -- (-3, -6, 0);
%%
%%
%% Drawing the vertices in the front
%%
\node[vertex] at (0, -6, 3)     {};
\node[vertex] at (0, 6, 3)     {};
\node[vertex] at (6, 0, 3)     {};
\node[vertex] at (6, -3, 0)     {};
\node[vertex] at (3, 6, 0)     {};
\node[vertex] at (6, 3, 0)     {};
\node[vertex] at (3, -6, 0)     {};
\node[vertex] at (3, 0, 6)     {};
\node[vertex] at (0, -3, 6)     {};
\node[vertex] at (0, 6, -3)     {};
\node[vertex] at (0, 3, 6)     {};
\node[vertex] at (-3, 6, 0)     {};
\node[vertex] at (-3, 0, 6)     {};
\node[vertex] at (-6, 3, 0)     {};
\node[vertex] at (-6, 0, 3)     {};
\node[vertex] at (-6, -3, 0)     {};
\node[vertex] at (-6, 0, -3)     {};
\node[vertex] at (-3, -6, 0)     {};
\end{tikzpicture}
\caption{3-permutohedron}
\label{fig:pi3}
\end{subfigure}
\begin{subfigure}{.3\textwidth}
\centering
\input{tikz/triangle_up.tex}
\caption{Another triangle}
\label{fig:fail}
\end{subfigure}
\caption{Symmetrizing combinatorially equivalent polytopes can lead to combinatorially different polytopes. }
\label{fig:intro}
\end{figure}

The combinatorics of a polytope is captured by its face poset, and the normal fan of a polytope contains not only its combinatorics but also some of its geometric information. Our first main result, Theorem \ref{thm:fundamental_cones}, describes the normal fan of $ \fG(P) $ in terms of the normal fan of $ P $, as long as $ P $	satisfies certain properties. 
The main feature of finite reflection groups that we exploit is the existence of a {fundamental chamber} $\Phi$.
This is a cone that controls the geometry of the group action.
To describe the combinatorics of the symmetrization we need to understand, very precisely, how the fundamental chamber is subdivided by the normal fan of $ P $.
This lead us to the definition of the ``fundamental fan'' and its ``refined'' version, expanded in Section \ref{sec:fundamental}.
The combinatorics of the refined fundamental fan is described by a decorated poset that gives enough information to construct the face poset of the resulting symmetrization.
Theorem \ref{thm:symposet} allows us to go from the refined fundamental fan of $ P $ to the normal fan of $ \fG(P) $.

The main application of this article is an instance of the realization problem, that is, determine whether a given poset is the face poset of a polytope.
A \emph{$ \fG $-symmetric poset} is a poset with a $ \fG $-action on the elements that preserves the partial order.  
Theorem \ref{thm:sym-poset} reduces the realization problem of a $\fG$-symmetric poset to a problem of realizing its generating subposet.  
We give an example of how to take advantage of our result to solve such a realization problem in the last section. 
We realize a poset that is a \emph{hybrid} between the poset of ordered set partitions and the Boolean poset. 

In the literature, there have been instances of polytopes where two combinatorial structures are combined to form a new one.  
The face posets of some of them are symmetric posets. 
Typically, this involves having polytopes $P$ and $Q$	where each face of $P$ is modified using a structure derived from $Q$.
One recent example is the bipermutohedron \cite{ardila2020bipermutahedron}.
This polytope is constructed at the level of fans by subdividing the chambers of the normal fan of the diagonal simplex with carefully placed copies of the braid fan, and its face poset is $\fS_d$-symmetric.
Another example of a symmetric poset that can be realized as a polytope is the permuto-associahedron, first constructed by Kapranov \cite{kapranov} as a hybrid of the face poset of the permutohedron and that of an associahedron. 
It was then realized as polytopes independently by Reiner-Ziegler \cite{reiner1994coxeter}, Gaiffi \cite{gaiffi}, and the authors \cite{castillo2023permuto}.
A different hybrid between permutohedron and associhedron is constructed by Barali\'c-Ivanovi\'c-Petri\'c \cite{baralic_ivanovic_petric} both in the poset setting and in the polytope setting.

The original motivation of this paper comes from the observation that all aforementioned constructions for permuto-associahedron can be viewed as the $\fS_d$-symmetrization of an associahedron. 
However, the goal of this paper is twofold. 
On one hand, we want to build a general theory on $\fG$-symmetrization of a polytope (satisfying some conditions) with an arbitrary reflection group $\fG$ that unifies these constructions. 
One the other hand, since ``combining two structures'' is not a formal definition, it would be interesting to establish formally when one can say a poset (a polytope, respectively) is a \emph{hybrid} of two other posets (polytopes, respectively).
As the $\fG$-symmetrization of a polytope $P$ can be considered as the hybrid between a $\fG$-permutohedron and the polytope $ P $, studying the connection of the face poset of $\fG(P)$ to the combinatorics of the fundamental chamber $\Phi$ associated to $\fG$ and the face poset of $P$ will give us hints on what might be a correct definition of a \emph{hybrid} poset, at least in this particular setting (see Definition \ref{defn:hybrid} and Remark \ref{defn:hybrid}). 
We remark that although the face poset of the bipermutohedron \cite{ardila2020bipermutahedron} is $\fS_d$-symmetric, it does not fit in our framework of $ \fS_{d} $-symmetrizations, and thus it is a not a \emph{hybrid} in the sense we define in this paper. 

In an upcoming paper, we will explore the connection between the type-A symmetrization of a polytope and its type-B symmetrization, and show how to construct its type-B symmetrization from its  type-A symmetrization.

\subsection*{Organization of the paper}
In Section \ref{sec:prelim} we review some basic results and fix our notation.
In Section \ref{sec:sym-polytope} we develop the basic results about symmetrizing a polytope.
In Section \ref{sec:fundamental} we focus on the combinatorics of fundamental fans.
In Section \ref{sec:symm-poset} we describe symmetrization in terms of posets.
Finally in Secion \ref{sec:application} we show an extended example of how to solve a realization problem using our tools.

\subsection*{Acknowledgements}
We thank Jean-Philippe Labb\'e for helpful conversations.
We also thank the Mathematics Department at UC Davis for the hospitality during the summer of 2023.
Federico Castillo is partially supported by FONDECYT Regular Grant \#1221133. 
Fu Liu is partially supported by a grant from the Simons Foundation \#426756 and an NSF grant \#2153897-0.

\section{Preliminaries}
\label{sec:prelim}

\subsection{Polytopes and normal fans}

We begin by reviewing basic notions of polyhedral geometry.
A general reference is \cite{ewald}.
Let $V \subseteq \R^d$ be a subspace of the $d$-dimensional Euclidean space and $W$ is the dual space of $V$ that contains all linear functionals on $V$, 
and suppose the perfect pairing between $V$ and $W$ is $\langle \cdot, \cdot \rangle: W \times V \to \R$. 

In this paper, we will assume that $[ \cdot , \cdot ]$ is an inner product on $\R^d$ (which will be the dot product on $\R^d$ for all the examples we have).
This naturally gives a way to define the dual space $W$ of $V$ as a quotient space of $\R^d.$ 
More precisely, $W = \R^d/V^\perp$, where $V^\perp = \{ \x \in \R^d \ : \ [\x, \v] = 0, \ \forall \v \in V\}$ is the \emph{orthogonal complement} of $V$, 
and the perfect pairing between $V$ and $W$ is induced by the inner product $[ \cdot , \cdot ]$. In the case when $V \neq \R^d$ and thus $V^\perp$ is nontrivial, each vector $\w \in W$ is an equivalent class of vectors in $\R^d$. Abusing the notation, we always use one of the vectors in this equivalent class to represent $\w$, called a \emph{representative} of $\w$.

Let $U \subseteq \R^d$ be an affine space in $\R^d$ that is a translation of $V,$ that is, there exists $\x \in \R^d$ such that $U = V + \x.$
One sees that there exists a unique $\x_0 \in V^{\perp}$ such that $U = V + \x_0;$ in this case, we will write $U=V(\x_0).$ We may consider $U=V(\x_0)$ and $W$ a pair of dual spaces with perfect pairing induced by the pairing between $V$ and $W$, which is induced by the inner product on $\R^d$. Abusing the notation, we still use $\langle \ , \ \rangle$ to denote the pairing between $U$ and $W$. Hence, for any $\u \in U$ and $\w \in W$, we have
\[ \langle \w, \u \rangle \defeq \langle \w, \u-\x_0\rangle = [ \w, \u-\x_0].\]

A \emph{polytope} in $U$ is the convex hull of a finite set in $U$. We often call an $e$-dimensional polytope an \emph{$e$-polytope}. 
A \emph{polyhedron} is the intersection of finitely many closed halfspaces.
By Minskowsi-Weyl Theorem \cite[Theorem 1.4]{ewald} a polytope is a bounded polyhedron.
A \emph{face} of a polyhedron $P$ in $U$ is a subset $F\subseteq P$ such that there exists $\w\in W$ such that \[F=\left\{\x\in P~:~ \langle \w, \x \rangle \geq \langle \w, \y \rangle, \quad \forall \y \in P \right\}.\]
Notice that $ P $ is a face of itself, by considering $ \w = \mathbf{0} $.
In addition we consider $\emptyset$ to be a face of $P$ (of dimension -1). 
We denote by $ f_{i}(P) $ the number of $ i $-dimensional faces of $P$, and call $(f_0, f_1, \dots, f_{e-1}, f_e)$ the \emph{f-vector} of $P$. 
In this paper, when we discuss interior of a polyhedron $P$, denoted by $P^\circ$, we always mean the relative interior of $P$.

A \emph{homogeneous cone} is a polyhedron defined by homogeneous linear inequalities. Alternatively, a homogeneous cone $K$ can be defined as the \emph{conic hull} of a set $R=\{ \r_1,\dots, \r_k\}$ of vectors: 
\[ K=\cone( \r_1, \dots, \r_k) \defeq \left\{ \sum_{i=1}^k \lambda_i \r_i \ : \ \lambda_i \ge 0\right\}.\]
We call any such a set $R$ is a \emph{generating set} for $K$, and an inclusion-minimal generating set is a \emph{basis} for $K$. 
A \emph{(polyhedral) cone} is a translation of a homogeneous cone. 
A cone is \emph{pointed} if it contains a $0$-dimensional face. 
It is well-known that for a pointed cone $K$, a set $\{ \r_1, \dots, \r_k\}$ of vectors is a basis for $K$ if and only if $\{ \cone(\r_i) \ : \ 1 \le i \le k\}$ is the set of rays, i.e., one-dimensional faces, of $K.$ A \emph{simplicial} cone is a pointed cone $K$ generated by an independent set of vectors, which is equivalent to the number of rays of $K$ equals to the dimension of $K$.

A \emph{fan} is a collection $\Sigma$ of cones that is a simplicial complex, i.e., if $\sigma \in \Sigma$, then any face of $\sigma$ is in $\Sigma$, and if $\sigma_1, \sigma_2 \in \Sigma,$ then $\sigma_1 \cap \sigma_2$ is a face of both.
The maximal dimensional cones of a fan are called \emph{chambers}, and the \emph{dimension} of a fan is the dimension of its maximal dimensional cones. 
We denote by $ f^{i}(\Sigma) $ the number of codimension $ i $ cones in $ \Sigma $.
The \emph{support} of a fan is the union of its cones.
A fan is \emph{complete} if its support is the entire space.

A \emph{conic dissection} of $W$ is a collection of full-dimensional cones such that the union is $W$ and different cones have disjoint interiors. 
Note that the chambers of a complete fan form a conic dissection.
A fan or conic dissection is \emph{pointed} (resp. \emph{homogeneous}) if all of its cones are pointed (resp. homogeneous). 
All the fans in $W$ and conic dissections of $W$ we consider in this paper are pointed and homogeneous, and all the fans in $U=V(\x_0)$ are pointed with $\x_0$ being their $0$-dimensional faces. 

The set of all nonempty faces of $P$ partially ordered by inclusion forms the \emph{face poset\footnote{The usual definition of face poset includes the nonempty set as the minimum element. But in this paper, we use the convention of not including the empty face.} } $\FL(P)$ of $P$. 
Note that the face poset of an $e$-polytope is graded of rank $e$.  
Likewise, the set of all cones in a fan $ \Sigma $ partially ordered by inclusion forms the \emph{face poset} $ \FL(\Sigma) $ of $ \Sigma $.
If two polytopes or fans have isomorphic face posets, we say they are \emph{combinatorially equivalent}; otherwise, we say they are \emph{combinatorially different}.

\begin{defn}\label{defn:normal}
Suppose $U, V, W$ are given as above, and $P\subseteq U$ is a polytope.
Given any a nonempty face $F$ of $P$, the \emph{normal cone} of $P$ at $F$ with respect to $W$ is defined to be
\[
\ncone_W(F, P) \defeq   \left\{ \w \in W \ : \ \langle \w, \y \rangle \geq \langle \w, \y' \rangle, \quad \forall \y \in F,\quad \forall \y' \in P \right\}.
\]
Therefore, $\ncone_W(F,P)$ is the collection of linear functionals $\w$ in $W$ such that $\w$ attains maximum value at $F$ over all points in $P.$ 
The \emph{normal fan} of $P$ with respect to $W$, denoted by $\Sigma_W(P)$, is the collection of all normal cones of $P$ as we range over all nonempty faces of $P$.
Since $P$ can be considered to be in any (affine) space $U$ that contains the affine span of $P$ and thus we can have difference choices of $W$. This is the reason why we include the subscript $W$.
We omit the subscript $W$ when it is clear which pair $(U,W)$ of dual spaces we use. 
\end{defn}

If a fan is the normal fan of a polytope we say it is \emph{polytopal}.
It is easy to see that $P$ is full-dimensional in $U$ if and only if $0 \in \Sigma(P),$ equivalently, all chambers in $\Sigma(P)$ are pointed.
Additionally for $P\subseteq U$ we define the \emph{perpendicular space} of $P$ to be
\begin{equation}\label{eq:Pvee}
P^{\vee} \defeq  \{\w\in W~:~\langle\w,\x\rangle=\langle\w,\y\rangle,\quad \forall \x,\y\in P\}.
\end{equation}
The subspace $P^\vee$ has positive dimension if and only if $P$ is not full-dimensional in $U$.

The following lemma summarize a few results on the relation between normal cones and faces of a polytope that can be verified directly from the definition. 
\begin{lem}\label{lem:face&fan} Let $P$ be a polytope.
    The $F \mapsto \ncone(F, P)$ induces an order-reversing bijection from %$\cF_{\ge 0}(P)$
    $\cF(P)$ to $\cF(\Sigma(P))$. That is,  
for any two faces $F, G$ of $P$, we have
\[ F \subseteq G \quad \text{if and only if} \quad \ncone(G, P) \subseteq \ncone(F,P).\]
Therefore, the poset %$\cF_{\ge 0}(P)$
$\cF(P)$ is isomorphic to the dual of the poset $\cF(\Sigma(P))$.

  Furthermore, if $F$ is an $k$-dimensional face of $P$, then $\ncone(F,P)$ is codimension $k$ in $\Sigma(P)$. Hence, for any $k,$
   \[ f_k(P) = f^k(\Sigma(P)).\]
\end{lem}

\begin{defn}\label{def:Boolean}
  The \emph{Boolean lattice} $ \cB_{d} $ is a poset that consists of all subsets of $ [d] $ ordered by inclusion.
We define $ \widehat{\cB}_{d} $, the \emph{truncated Boolean poset} to be the poset obtained from the Boolean lattice $ \cB_{d} $ by removing the empty set.
\end{defn}

\begin{ex}\label{ex:simplex}
  Let $U=V=W=\R^d$.  
	Recall the \emph{standard simplex} $\Delta_{d-1} \in \R^d$ is the convex hull of the standard basis $\{ \e_1, \e_2, \dots, \e_d\}$ of $\R^d$:
  \[ \Delta_{d-1} = \conv\{ \e_1, \e_2, \dots, \e_d\}.\]
  It is well-known that the map that sends each nonempty $B \subseteq [d]$ to $\Delta_{B} := \conv\{ i \in B \ : \ \e_i\}$ is a poset isomorphism from $\widehat{\cB}_{d}$ to the face poset $\cF(\Delta_{d-1})$ of the standard simplex $\Delta_{d-1}$. 
	Moreover, for each nonempty $B \subseteq [d]$, 
	\begin{equation}\label{eq:ncone_simplex} \ncone(\Delta_B, \Delta_{d-1}) = \sigma_{d,B} := \left\{ \y \in \R^d  \ : \ 
    \begin{array}{c} y_i = y_j \text{ if $i,j \in B$} \\
y_i \le y_j \text{ if $i \notin B$ and $j \in B$} \end{array} \right\}.\end{equation}
\end{ex}

\subsection{Finite reflection groups}\label{subsec:reflection}

Recall that $V$ is a subspace of $\R^d$ which is equipped with an inner product $[ \cdot , \cdot ]$, $U= V(\x_0) = V + \x_0$ is the translation of $V$ by a vector $\x_0 \in V^\perp$ and $W = \R^d/V^\perp$. We consider $U$ and $W$ are dual spaces with perfect pairing $\langle \ , \ \rangle: W \times U \to \R$ defined by
\[ \langle \w, \u \rangle = [ \w, \u - \x_0].\]
This allows us to identify points in $U$ and $W$: for any $\u \in U$, its corresponding point in $W$ is $\u^*\defeq (\u+V^\perp)/V^\perp;$ and for any $\w \in W$, its corresponding point in $U$ is the unique representative $\w^*$ of $\w$ that lies in $U.$

A \emph{hyperplane arrangement} is a finite set $\mathcal{H}=\{H_1,\cdots,H_m\}$ of affine hyperplanes in $ W $. 
In this paper, we assume that the intersection of all hyperplanes in $\cH$ is a single point.
Then the set $\mathcal{H}$ induces a pointed homogeneous polyhedral fan on $W$ that we denote $\Sigma(\mathcal{H})$.

Let $t_i\in \mathfrak{Gl}(W)$ be the reflection associated to the hyperplane $H_i$.
By definition the group $\mathfrak{G}\subset\mathfrak{Gl}(W)$ generated by $\{t_1,\dots,t_m\}$ acts on $W$, and dually on $U$. 
To be more precise, assume $\fG$ acts on the left on $W$, then there is a canonical way to define a right action of $\fG$ on $U$ such that
\begin{equation}
\langle g\w ,\u\rangle = \langle \w,\u g \rangle, \quad \text{for all $g \in \fG$, $\u \in U$, $\w \in W$}.
\label{eq:dualaction}
\end{equation}

We will focus on hyperplane arrangements that have the special property that $\mathfrak{G}$ is finite.
In this case $\mathfrak{G}$ is called a \emph{finite reflection group} and the set $\mathcal{H}$ and the group $\mathfrak{G}$  determine each other, see \cite[Section 1.14]{humphreys}. 

Let $H_i^*$ be the hyperplane in $U=V(\x_0)$ that identifies $H_i.$ Then $\mathcal{H}^*\defeq\{H_1^*,\dots,H_m^*\}$ is a hyperplane arrangement in $U$, which defines a corresponding pointed polyhedral fan $\Sigma(\cH^*)$ (with $\x_0$ being its $0$-dimensional face). % such that $\mathfrak{G}$ is also the reflection group but now acting on the left.
We arbitrarily choose a chamber $\Phi$ of $\Sigma(\cH)$ and its corresponding chamber $\Psi = \Phi^*$ of $\Sigma(\mathcal{H}^*)$ and denote them the \emph{fundamental chambers}. For convenience, we call $(\Phi, \Psi)$ a pair of fundamental chambers associated to the hyperplane arrangement $\cH$. 
They will play a central role in what follows. 
For any $g \in \fG$, since the action of $g$ is an isomorphism on both $W$ and $U,$ we have that both $g\Phi $ and $ \Psi g$ are pointed cones. 
However, we have more than that \cite[Secion 1.12]{humphreys}: 
\begin{lem}\label{lem:Gbij}
The correspondences $g\mapsto g\Phi $ and $g\mapsto \Psi g$ gives a bijection between $\mathfrak{G}$ and the set of chambers of $\Sigma(\mathcal{H})$ and of $\Sigma(\mathcal{H}^*)$  respectively.
\end{lem}

\begin{defnex}[type-A reflection group]\label{defn:typeA}
Let $[ \ , \ ]$ be the dot product on $\R^d$, and  $\balpha=(\alpha_1, \dots, \alpha_d)\in\R^d$.
Consider the space 
\[U^\balpha_{d-1} \defeq \left\{ \x \in \R^{d} \ : [\1, \x] = [\1, \balpha] = \sum_{i=1}^d \alpha_i \right\} \subseteq \R^{d}\]
where $\1 = (1, 1, \dots, 1)$ denotes the all-one vector in $\R^d$.
Then $U^\balpha_{d-1}$ is a translation of the vector space $V_{d-1} = \left\{ \x \in \R^{d} \ : [\1, \x] = 0 \right\}.$ 
We have $V_{d-1}^\perp$ is spanned by $\1$,
and the dual space of $U^\balpha_{d-1}$ is $W_{d-1}: = \R^d/\1$.
  The vector space $V_{d-1}$ has a canonical basis $ \left\{ \mathbf{a}_{1}, \dots, \mathbf{a}_{d-1} \right\} $ defined by
$\mathbf{a}_{i} = -\ee_{i}+\ee_{i+1},$
and $W_{d-1}$ has a canonical basis $\left\{\f_{1}^{(d)}, \dots, \f_{d-1}^{(d)} \right\}$ defined by
\begin{equation}\label{eq:defnfi}
		\f_{i}^{(d)} \defeq (\underbrace{0, \dots, 0,}_{i} \underbrace{1, \dots, 1}_{d-i}).
	      \end{equation}
Note that $W_{d-1}$ is a quotient space, and thus its basis is formed by quotient vectors. The vectors $\f_i^{(d)}$'s are just representatives of these quotient vectors. 
Often, when the length of the vector $\f_i^{(d)}$ is clear from the context, we simply write $ \f_{i} $. 
Finally, it is straightforward to verify that $\langle \a_i, \f_i^{(d)} \rangle = \delta_{i,j}.$ Hence, these two bases $\{ \a_i\}$ and $\{ \f_i^{(d)} \}$ are dual to one another.

We define the hyperplane arrangement 
\begin{equation}\label{eq:typeA}
\mathcal{A}_{d-1}\defeq  \{H_{ij}~:~1\leq i<j\leq d\}, \quad \text{ where }  H_{ij}\defeq  \{\w\in W_{d-1}~:~w_i=w_j\}.
\end{equation}
The arrangement $\mathcal{A}_{d-1}$ is called the \emph{braid arrangement} and its fan $\Sigma(\mathcal{A}_{d-1})$ the \emph{braid fan}. Its corresponding arrangement in $U$ is 
\[\cA_{d-1}^* = \{ H_{i,j}^* \ : \ 1 \le i < j \le d \}, \quad \text{ where } H_{ij}^* \defeq  \{\u\in U^\balpha_{d-1}~:~u_i=u_j\}.\]

The reflection group associated to $\cA_{d-1}$ is isomorphic to the symmetric group $\fS_{d}$, the group of bijective functions $\pi:[d]\to[d]$ under composition, 
and it acts on the left on $ W_{d-1} $ as follows: for each $\pi \in \fS_d$ and $\w=(w_1, \dots, w_d) \in W_{d-1} ,$
\[
\pi(w_1,\dots,w_d)  =(w_{\pi^{-1}(1)},\dots,w_{\pi^{-1}(d)}).
\]

In particular, if we let $t_{i,j}$ be the reflection associated to $H_{i,j}$, then $t_{i,j}$ corresponds to the transposition $(i,j) \in \fS_d$, whose action swaps the $i$th and $j$th coordinate of points in $W_{d-1}$.
It is easy to check that the corresponding right action of $\fS_d$ on $U^{\balpha}_{d-1}$ is given by 
\[
(u_1,\dots,u_d)\pi  =(u_{\pi(1)},\dots,u_{\pi(d)}).
\]
Finally, we choose the fundamental chambers for $\mathcal{A}_{d-1}$ (resp. $\cA_{d-1}^*$) to be
\small
\[
\Phi_{d-1} = \{\w\in W_{d-1}~:~w_1\leq w_2\leq \dots\leq w_d\}\qquad   (\text{resp. } \Psi_{d-1} = \{\u\in U^\balpha_{d-1}~:~u_1\leq u_2\leq \dots\leq u_d\} ),
\]
\normalsize
in other words, the set of all vectors in $W_{d-1}$ (resp. $U^\balpha_{d-1}$) that with weakly increasing entries.
Its interior $\Phi_{d-1}^\circ$ (resp. $\Psi_{d-1}^\circ$) consists of all vectors in $W_{d-1}$ (resp. $U^\balpha_{d-1}$) with \emph{strictly} increasing entries.
Notice that the fundamental chamber $\Phi_{d-1}$ is a simplicial cone with the basis $\left\{\f_{1}^{(d)}, \dots, \f_{d-1}^{(d)} \right\}$.

\end{defnex}

\subsection{Ordered set partitions } 
For convenience, we use $\Sigma(\Phi)$ to denote the fan that is induced by $\Phi$, which is the set of nonempty faces of $\Phi$. 
It turns out that understanding how cones in $\Sigma(\Phi)$ interact with the fan of the polytope that we will symmetrize is essential in our analysis in Sections \ref{sec:fundamental} and \ref{sec:symm-poset}. 
Since our main examples involve type-A reflection group, we will describe below in detail the face posets of $\Sigma(\cA_{d-1})$ and $\Sigma(\Phi_{d-1})$, which have close connection to the poset of ordered set partitions

\begin{defn}
\label{defn:ordered_set_partitions}
An \emph{ordered (set) partition} $ \cS $ of $ [d] $ is an ordered list of pairwise disjoint subsets (called \emph{blocks}) $(S_{1}, \dots, S_{k})$ such that $ \bigsqcup S_i = [d] $.
We define the \emph{type set} of $\cS=(S_{1}, \dots, S_{k})$, denoted by $\Type(\cS)$, to be the set 
$ \left\{ \sum_{j=1}^i |S_j| \ : \ 1 \le i \le k-1 \right\} \subseteq [d-1]$.

We denote by $ \mathcal{O}_{d} $ the set of ordered partitions of $ [d] $, 
and by $ \mathcal{O}^{i}_{d} $ the set of ordered partitions of $[d]$ with exactly $ i $ blocks.

Furthermore, we define a partial order $\leq$ on $ \mathcal{O}_{d} $ by declaring  $ \mathcal{S}_{1} \leq \mathcal{S}_{2} $ if $ \mathcal{S}_{1} $ is a refinement of $ \mathcal{S}_{2} $.

We say $ \cS=(S_1,\dots,S_k) \in \cO_{d} $ \emph{standard}, if every element of $ S_i $ is smaller than every element in $ S_{i+1} $ for $ i=1,\dots,k-1 $. 
We denote by $\sO_d^\std$ the set of all standard ordered set partitions in $\sO_d$.
\end{defn}

It is easy to see that the $  \mathcal{O}_d $ is a graded poset of rank $d-1$, and the rank of an ordered set partition $\cS$ is exactly $ d - |\cS| $, where $ |\cS| $ is the number of blocks in $\cS$. 
Hence, for $1 \le i \le d$, the set $\cO_d^{i}$ consists of all elements of $  \mathcal{O}_d $ of rank $d-i.$
In particular, the only element $([d])$ in $  \mathcal{O}_d^1 $ is the unique maximal element of $\cO_d$, and the set $\cO_d^d$ consists of all the $d!$ minimal elements of $\cO_d$, which are in bijection with $\fS_d$.  

\begin{defn}\label{def:osp_cone}
  For any $\cS=(S_1,\dots,S_k) \in \cO_{d}$, we define
	\begin{equation}\label{eq:osp_cone_alt} \sigma_\cS = \left\{ \w \in W_{d-1} \ : \begin{array}{c} w_i = w_j \text{ if $i,j \in S_{a}$ for some $ a $} \\
			w_i \le w_j \text{ if $i\in S_a, j\in S_b$ with $a < b$} \end{array} \right\}.\end{equation}
\end{defn}

The cones $ \sigma_\cS $ are part of a bigger family of preorder cones.

\begin{ex} 
  Let $\cS = (123, 45, 6 ,78) \in \cO_{8} $. 
	We have that $ \Type(\cS) = \left\{ 3,5,6 \right\} $.
  The cone  $\sigma_\cS$ in $W_7$ defined by 
  \[\{ \w \in W_7 ~:~ w_1 = w_2 = w_3 \leq w_{4} = w_{5} \leq w_{6} \leq w_{7} = w_{8} \}.\] 
\end{ex}

See \cite[Section 3]{postnikov2006faces} for more information and in particular the following lemma. %, recalling that $\cF^\star$ is the dual of a poset $\cF$. 

\begin{lem}\label{lem:faces_Phi}
  The map $\cS \mapsto \sigma_\cS$ gives a poset isomorphism from $\cO_d$ to the dual of the poset $\cF(\Sigma(\cA_{d-1}))$.
Furthermore, it induces a poset isomorphism from $\cO_d^\std$ to the dual of the poset $\cF(\Sigma(\Phi_{d-1}))$.
\end{lem}

\begin{ex}
	Let $d=4$. 
	There are $8$ standard ordered set partitions, corresponding to the $8$ nonempty faces of $\Phi_3$. 
	Figure \ref{fig:labels} depicts an affine slice of the fundamental chamber $\Phi_3$ in which the affine slices of the $7$ positive dimensional faces of $ \Phi_{3} $. % are labeled as $\sigma_\cS$'s. 
	The unique zero dimensional cone (the origin) corresponding to the trivial partition $([4])$ is not shown as it does not intersect any affine slice.

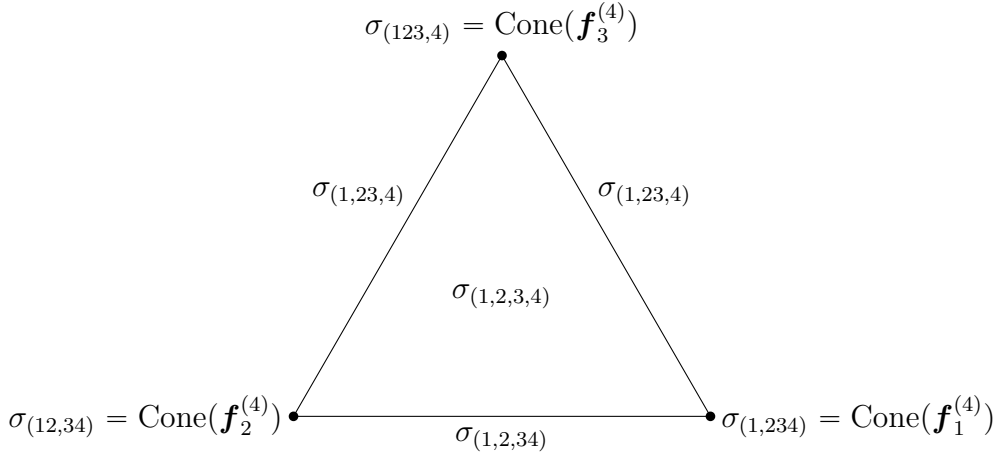
\begin{figure}[ht]
\begin{tikzpicture}[
vertex/.style={inner sep=1pt,circle,draw=black,fill=black,thick,anchor=base},
edge/.style={color=purple, thick},
normal/.style={color=black},
boundary/.style={color=black, dotted},
scale = 0.8
]					

\coordinate (0001) at (90:4);
\coordinate (0011) at (210:4);
\coordinate (0111) at (330:4);
\coordinate (0122) at (270:2);

\node[vertex] at (0001) {};
\node[vertex] at (0011) {};
\node[vertex] at (0111) {};

\draw[normal] (0001)--(0011)--(0111)--cycle;

\node at (0:0) {$\sigma_{(1,2,3,4)}$};
\node[right] at (50:2.2) {$\sigma_{(1,23,4)}$};
\node[left] at (130:2.2) {$\sigma_{(1,23,4)}$};
\node[above] at (0001) {$\sigma_{ (123,4)} = \cone(\f_3^{(4)}) $};
\node[left] at (0011) {$\sigma_{ (12,34)} = \cone(\f_2^{(4)}) $};
\node[right] at (0111) {$\sigma_{ (1,234)} = \cone(\f_1^{(4)}) $};
\node[below] at (0122) {$\sigma_{ (1,2,34 )}$};

\end{tikzpicture}
\caption{The positive dimensional faces of the fundamental chamber $\Phi_3$.}
\label{fig:labels}
\end{figure} 
\end{ex}

\subsection{Symmetric posets and generators} Suppose $\cT$ is a poset such that $\fG$ acts on it. Recall that we say $\cT$ is a \emph{$ \fG $-symmetric poset} is the $ \fG $ action preserves the partial order. We are interested in its subposets that generate it: 

\begin{defn}
	\label{defn:symmetric_poset}
	Suppose $\cT$ is a $\fG$-symmetric poset and $\cZ$ an induced subposet of $\cT$.
	We say that $\cZ$ is a \emph{generating subposet} or a \emph{$\fG$-generator} for $\cT$
	if the following conditions are satisfied:
  \begin{enumerate}
    \item %For each $t \in \cT$, there exists a unique $z \in \cZ$ such that $t = g z$ for some $g \in \fG$. In other words, 
      We have the following disjoint decomposition for $\cT$ as a set:
      \[ \cT = \bigsqcup_{z \in \cZ} \fG \cdot z.\]

    \item Suppose $t_i=g_i z_i$ for some $g_i \in \fG$ and $z_i \in \cZ$, for each $i=1,2$. Then $t_1 \le_{\cT} t_2$ if and only if $z_1 \le_{\cZ} z_2$ and there exists $g \in \fG$ such that $t_i = g z_i$ for $i=1,2$. 
  \end{enumerate}

\end{defn}

\begin{ex}\label{ex:Boolean_symmetric}
An example of a $ \fS_{d} $-symmetric poset is the truncated Boolean poset $\widehat{ \cB}_{d} $ (Definition \ref{def:Boolean}). 
An element $ \pi \in \fS_d $ is a function on $ [d] $ and so it acts on the subsets by acting directly on each element.
A generating subposet is the chain $ \left\{ 1 \right\} \leq \left\{ 1,2 \right\} \leq  \dots \leq \left\{ 1,2,  \dots ,d \right\} $.
\end{ex}

\begin{ex}
	Similar to Example \ref{ex:Boolean_symmetric}, the poset $ \cO_d $ of ordered partitions is $ \fS_{d} $-symmetric.
  A generating subposet is the induced subposet $\cO_d^\std$ of standard ordered partitions. 
\end{ex}

\section{Symmetrizing polytopes} \label{sec:sym-polytope}

In this section, we develop the idea of symmetrization of polytopes in the most general form.
From now on, we always assume $\mathfrak{G}\subseteq \Gl(W)$ is a finite reflection group arising from (or determining) the hyperplane arrangement $\cH$ in $W$, where $\Sigma(\cH)$ is a pointed homogeneous polyhedral fan, and that $(\Phi, \Psi)$ is a pair of fundamental chambers associated to $\cH.$

\begin{defn}
For any polytope $P\subseteq V$, we define its \emph{$\mathfrak{G}$-symmetrization} to be the polytope
\begin{equation}
\mathfrak{G}(P)=\conv\{\x g~:~g\in\mathfrak{G}, \x\in P\}.
\end{equation}	
\end{defn}

\begin{ex}\label{ex:permutohedron}
The most basic $\mathfrak{G}$-symmetric polytope comes from  symmetrizing a single point.
In the case of type-A reflection group $\fS_d$, given any vector $\u=(u_1,\dots,u_d) \in U_{d-1}^\balpha$, the $\fS_d$-symmetrization of $\{\u\}$ is known as the permutohedron $P_d(u_1,\dots,u_d)$ in \cite[Definition 2.1]{postnikov2009permutohedra}.
For a general reflection group $\mathfrak{G}$, the symmetrization $\mathfrak{G}(\u)$ of an aribitrary $\u \in U$ is known in the literature as an orbit polytope \cite{hohlweg2011permutahedra}, and the symmetrization $\fG(\u)$ of $\u \in \Psi^\circ$ is called a \emph{$\fG$-permutohedron}.
\end{ex}

\begin{prop}{\cite[Section 3.1]{hohlweg2011permutahedra}}
\label{prop:g-perm}
Suppose $\u \in\Psi^\circ$. 
Then for each $g\in \mathfrak{G}$, the point $\u g$  is a vertex of $\mathfrak{G}(\u)$ and $\ncone(\u g,\mathfrak{G}(\u))=g^{-1}\Phi $.
As a consequence, the fan $\Sigma(\mathcal{H})$ is the normal fan of {the $\fG$-permutohedron} $\mathfrak{G}(\u)$. Since $\Sigma(\mathcal{H})$ is pointed, $\fG(\u)$ is full-dimensional.
\end{prop}

\begin{ex}\label{ex:g-perm}
Consider the braid arrangement $ \mathcal{A}_{d-1} $ and its corresponding reflection group $ \fS_{d} $. Let $\u \in \Psi_{d-1}^\circ.$ Then $\fS_d(\u)$ is a $\fS_d$-permutohedron of dimension $d-1$. The normal cone of $\fS_d(\u)$ at the vertex $\u$ is the fundamental chamber $\Phi_{d-1}$. An example for $d=3$ is shown in Figure \ref{fig:normal_cone}.

Moreover, we have that $\Sigma(\cA_{d-1})$ is the normal fan of $\fS_d(\u)$. Therefore, it follows from Lemmas \ref{lem:face&fan} and \ref{lem:faces_Phi} that the face poset $\cF(\fS_d(\u))$ of the $\fS_d$-permutohedron $\fS_d(\u)$ is isomorphic to the poset $\cO_d$ of ordered set partitions of $[d]$.

\begin{figure}[ht]
\begin{tikzpicture}
\begin{scope}[xshift=-3cm]
\draw(240:2cm)--(60:2cm) (180:2cm)--(0:2cm) (120:2cm)--(300:2cm);
\draw (0:2cm)--(0:0cm)--(60:2cm);
\draw[fill=white!75!black, draw=none] (0:2cm) arc (0:60:2);
\draw[fill=white!75!black, draw=none](0:0cm) -- (0:2cm) -- (60:2cm) -- cycle;

\node at (30:2.2cm) {$\Phi_{2}$};
\end{scope}

\begin{scope}[xshift=+3cm]
\draw[dashed] (240:2cm)--(60:2cm) (180:2cm)--(0:2cm) (120:2cm)--(300:2cm);

\filldraw[black] (50:1.5cm) circle (1.5 pt);

\node[right] at (50:1.5cm) {$\u$};
\node at (30:2.2cm) {$\Psi_{2}$};

\draw (70:1.5cm) -- (170:1.5cm) -- (-170:1.5cm) -- (290:1.5cm) -- (-50:1.5cm) -- (50:1.5cm)--cycle;

\draw[fill=black] (70:1.5cm) circle (1.5 pt);
\draw[fill=black] (170:1.5cm) circle (1.5 pt);
\draw[fill=black] (-170:1.5cm) circle (1.5 pt);
\draw[fill=black] (290:1.5cm) circle (1.5 pt);
\draw[fill=black] (-50:1.5cm) circle (1.5 pt);

\end{scope}

\end{tikzpicture}
\caption{An example of a permutohedron (in the right) and its normal fan (in the left). The fundamental chamber $ \Phi_{2} $ is highlighted.}
\label{fig:normal_cone}
\end{figure}
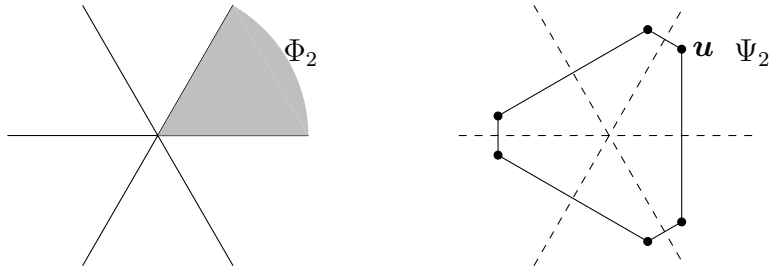
\end{ex}

The line segment introduced below will be our running example to illustrate the concepts as we introduce them.
\begin{ex}\label{ex:running_example}
Consider the braid arrangement $ \mathcal{A}_{3} $ and its corresponding reflection group $ \fS_{4} $.
Let $S: = \conv \left\{(1,2,6,8), (0,4,5,8)\right\}$ be a line segment in $\R^{4} $. One checks that $S \subseteq U^{(1,2,6,8)}_{3}$. 
The $\fS_4$ symmetrization $ \fS_{4}(S) $ of $S$ is the three dimensional polytope depicted in Figure \ref{fig:running}.

\begin{figure}[ht]
\centering
\input{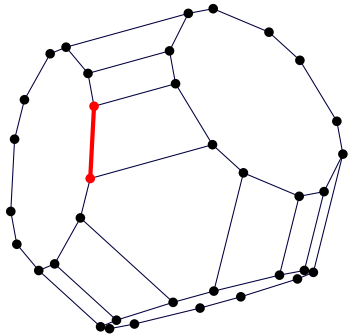}
\caption{Symmetrization of a line segment.}
\label{fig:running}
\end{figure}
\end{ex}

\begin{defn}\label{def:placed}
We say a polytope $P\subseteq U$ is \emph{$\Phi$-placed} if $P^\vee$ (see Equation \eqref{eq:Pvee}) and $\Phi^\circ$ have non-trivial intersection, that is, if there exists a $\bgamma \in \Phi^\circ$ such that the linear functional $\left\langle \bgamma,\cdot\right\rangle$ is constant on $P$.  
A full-dimensional polytope $P\subseteq U$ is not $\Phi$-placed, since $P^\vee = \{\mathbf{0}\}$.
\end{defn}

\begin{ex}\label{ex:running_placed}
The line segment $ S $ defined in Example \ref{ex:running_example} is $\Phi_3$-placed since the functional $(0, 1, 2, 3) \in \Phi_3^\circ$ is constant on the vertices of $S$ and hence is constant on the entire line segment $S$. 
\end{ex}

We now construct the collection of cones that are candidates for maximal cones in the normal fan of $\mathfrak{G}(P)$.

\begin{defn}
Suppose $P\subseteq U$ is a $\Phi$-placed polytope.
For each vertex $\v \in P$, we define
\[\kappa(\v,g)\defeq  g^{-1}\Phi\cap g^{-1}[\ncone(\v, P)].\]
We then define \[\KK_P\defeq \{\kappa(\v,g)~:~ g\in \mathfrak{G} \text{ and $\v$ is a vertex of }P\} \] to be the collection of these cones.
\end{defn}

Note that since $g$ is an isomorphism on $W,$ we have
\begin{equation}
\kappa(\v,g) = g^{-1}\Phi\cap g^{-1}[\ncone(\v, P)]=g^{-1}[\Phi\cap\ncone(\v, P)] 
\label{eq:kappa}
\end{equation}

\begin{lem}\label{lem:KKPdissect}
Let $P\subseteq U$ be a $\Phi$-placed polytope. 
The collection of cones $\KK_P$ is a pointed conic dissection of $W$. 
\end{lem}

\begin{proof}
For any $g \in \mathfrak{G}$ and any vertex $\v$ of $P,$ the cone $\kappa(\v,g)$ is contained in the pointed cone $\Phi g$, and hence is pointed as well. 
We will verify that $\KK_P$ satisfies the properties of being a conic dissection.

We first show that for each vertex $\v$ of $P$ and each $g \in \mathfrak{G}$ we have that the cone $\kappa(\v,g)$ is full-dimensional in $W.$
However, since $\kappa(\v,g) = g^{-1}[\Phi\cap\ncone(\v, P)]$,
it is enough to show that $\Phi \cap \ncone(\v,P)$ is full-dimensional. 
Because $P$ is $\Phi$-placed, there exists $\bgamma \in \Phi^\circ$ such that $\left\langle \bgamma,\cdot\right\rangle$ is constant on $P$, so $\bgamma \in \ncone(\v,P)$. 
It follows that $\Phi^\circ \cap \ncone(\v,P)$ is nonempty.
This together with the fact that both $\Phi$ and $\ncone(\v,P)$ are full-dimensional cones in $W$ implies that $\Phi \cap \ncone(\v,P)$ is full-dimensional. Therefore, $\KK_P$ is a set of full-dimensional cones in $W.$

To finish the proof, we need to verify that $\KK_P$ is a conic dissection, so we need to verify two facts: (1) the union of all cones is the entire space and (2) that different cones have disjoint interiors. 
Note that chambers in a normal fan form a conic dissection. 
This fact and Proposition \ref{prop:g-perm} imply that $\{ g^{-1} \Phi : g \in \mathfrak{G}\}$ and $\{ \ncone(\v,P)  :  \text{$\v$ is a vertex of $P$}\}$ are both conic dissections of $W$. Using these together with the definition of $\kappa(\v,g)$, one checks that $\KK_P$ has the desired properties.
\end{proof}

We remark that Lemma \ref{lem:KKPdissect} is not true if $P$ is not $\Phi$-placed assumption, as we see in the following example.

\begin{ex}
Consider the braid arrangement $\mathcal{A}_2$ and its corresponding reflection group $\fS_3$. See Figure \ref{fig:placed}: the left side is $\cA_2$ in $W_2$ and the right side is $\cA_2^*$ in $U_2^\balpha$. 
Let $P$ be the line segment on the right of the figure, 
then $P^\vee\subseteq W$ is as drawn on the left of the figure. 

The $\fS_3$-symmetrization of $P$ is the hexagon shown on the right of the figure. Notice that one vertex $\v$ of the line segment $P$ is not a vertex of the symmetrization, so $\kappa(\v,g)$ is trivial. 
\begin{figure}[ht]
\begin{tikzpicture}
\begin{scope}[xshift=-3cm]
\draw (240:2cm)--(60:2cm) (180:2cm)--(0:2cm) (120:2cm)--(300:2cm);
\draw[fill=white!75!black, draw=none] (0:2cm) arc (0:60:2);
\draw[fill=white!75!black, draw=none](0:0cm) -- (0:2cm) -- (60:2cm) -- cycle;

\draw[ultra thick] (161.4:2cm)--(341.4:2cm);
\node[right] at (341.4:2cm) {$P^\vee$};
\node at (30:2.2cm) {$\Phi_2$};
\end{scope}

\begin{scope}[xshift=+3cm]
\draw (240:2cm)--(60:2cm) (180:2cm)--(0:2cm) (120:2cm)--(300:2cm);

\draw[fill=white!75!black, draw=none] (0:2cm) arc (0:60:2);
\draw[fill=white!75!black, draw=none](0:0cm) -- (0:2cm) -- (60:2cm) -- cycle;
\filldraw[black] (30:0.8cm) circle (1.5 pt);
\filldraw[black] (50:1.5cm) circle (1.5 pt);
\draw (30:0.8cm) -- (50:1.5cm);

\node[left] at (30:0.8cm) {$\v$};
\node[left] at (25:1.5cm) {$P$};
\node at (30:2.2cm) {$\Psi_2$};

\draw[dotted] (70:1.5cm) -- (170:1.5cm) -- (-170:1.5cm) -- (290:1.5cm) -- (-50:1.5cm) -- (50:1.5cm)--cycle;

\draw[fill=white] (70:1.5cm) circle (1.5 pt);
\draw[fill=white] (170:1.5cm) circle (1.5 pt);
\draw[fill=white] (-170:1.5cm) circle (1.5 pt);
\draw[fill=white] (290:1.5cm) circle (1.5 pt);
\draw[fill=white] (-50:1.5cm) circle (1.5 pt);

\draw[fill=white] (90:0.8cm) circle (1.5 pt);
\draw[fill=white] (150:0.8cm) circle (1.5 pt);
\draw[fill=white] (210:0.8cm) circle (1.5 pt);
\draw[fill=white] (270:0.8cm) circle (1.5 pt);
\draw[fill=white] (-30:0.8cm) circle (1.5 pt);
\end{scope}

\end{tikzpicture}
\caption{A segment  $P$ in $U_2^\balpha$ that is not $\Phi$-placed.}
\label{fig:placed}
\end{figure}
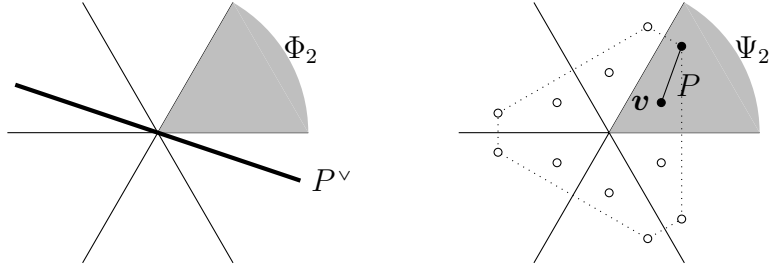
\end{ex}

\begin{defn}\label{def:appropriate}
Let $P\subseteq U$ be a  polytope.
We say $P$ is \emph {$\Psi$-appropriate} if $P\subseteq \Psi^\circ$. 
\end{defn}	

For example the line segment on the right of in Figure \ref{fig:placed} is $\Psi_{2}$-appropriate. 
Notice that to verify a polytope $P$ is $\Psi$-appropriate, it is enough to check that every vertex of $P$ is in $\Psi^\circ$.

\begin{ex}\label{ex:running_appropriate}
  The line segement $S \subseteq U_3^{(1,2,6,8)}$ defined in Example \ref{ex:running_example} is $\Psi_3$-appropriate since the entries in each vertex of $S$ are strictly increasing.
\end{ex}

We now present a generalization of Proposition \ref{prop:g-perm}.

\begin{thm}\label{thm:fundamental_cones}
Let $P$ be a $\Phi$-placed and $\Psi$-appropriate polytope with vertex set $\V$.
Then the following statements are true.
\begin{enumerate}
	\item \label{itm:fulldim} The vertex set of the symmetrization $\mathfrak{G}(P)$ is 
			\[ \mathfrak{G}\V \defeq   \left\{ \v g \ :\ (\v,g) \in  \V \times \mathfrak{G} \right\}.\]
		      \item \label{itm:vertexcone} For each $(\v,g) \in \V \times \mathfrak{G}$, we have $\ncone(\v g, \mathfrak{G}(P)) = \kappa(\v,g)$. (Recall that $\kappa(\v,g)$ is defined in \eqref{eq:kappa}.)
\end{enumerate}
Hence, $\mathfrak{G}(P)$ is a full-dimensional polytope in $U$.
\end{thm}

\begin{proof}
Using Lemma 2.4 of \cite{castillo2023permuto}, the proof of this proposition is reduced to the following claim:
For any distinct pairs $(\v,g), (\v',g') \in \V \times \mathfrak{G}$, we have that
\begin{equation}\label{eq:max_combined}
\langle \w, \v g\rangle > \langle \w,\v' g'\rangle\qquad \forall \w\in \kappa^\circ(\v,g). 
\end{equation}
We prove \eqref{eq:max_combined} by introducing an intermediate product and showing that for all $\w\in \kappa^\circ(\v,g),$
\begin{equation}
\langle \w, \v g\rangle \ge \langle \w,\v'g\rangle \ge \langle \w,\v'g'\rangle,  
\label{equ:intermprod2}
\end{equation}
where the first equality holds if and only if $\v = \v'$ and the second equality holds if and only $g = g'$.
Let $\z \defeq g\w .$ Then
\begin{equation}
\z = g\w  \in g\kappa^\circ(\v,g)  = \Phi^\circ \cap \ncone^\circ(\v,P).
\label{eq:expressu}
\end{equation}
Replacing $\w$ with $g^{-1}\z$ and applying \eqref{eq:dualaction}, we see that \eqref{equ:intermprod2} is equivalent to
\begin{equation}
\langle \z, \v \rangle \ge \langle \z,\v'\rangle \ge \langle \z,  \v'g'g^{-1}\rangle.
\label{equ:intermprod3}
\end{equation}
First, %since $\w \in \kappa^\circ(\v,g)=[\Phi^\circ_d \cap \ncone^\circ(\v,P)]g^{-1}$, we have that 
because $\z$ is in $\ncone^\circ(\v,P),$ 
we conclude that the first inequality in \eqref{equ:intermprod3} holds and the equality hold if and only if $\v = \v'.$ 
Next, since $P$ is $\Psi$-appropriate, the vertex $\v' \in \Psi^\circ$ and thus by Proposition \ref{prop:g-perm}, we have $\ncone(\v',\mathfrak{G}(\v'))=\Phi.$
Hence, $\z \in \Phi^\circ =\ncone^\circ(\v',\mathfrak{G}(\v'))$, and we conclude that the second inequality in \eqref{equ:intermprod3} holds and the equality hold if and only if $g = g'.$ 
\end{proof}

\begin{prop}\label{prop:Pface}
Let $P$ be a $\Phi$-placed and $\Psi$-appropriate polytope. 
Then $P$ is a face of $\fG(P)$.
\end{prop}

\begin{proof}
Since $ P $ is $ \Phi $-placed, there exists $ \bgamma \in \Phi^\circ$ such that the induced linear functional $\left\langle \bgamma,\cdot\right\rangle$ is constant on $ P $.
Furthermore, since $ P $ is $ \Psi $-appropriate, we have $P \subseteq \Psi^\circ.$ By Proposition \ref{prop:g-perm} for every vertex $ \v \in P $, the normal cone of $\fG(\v)$ at $\v$ is $\Phi.$ Hence, given that $\bgamma \in \Phi^\circ,$ we conclude that $\left\langle \bgamma, \v \right\rangle > \left\langle \bgamma, \v g \right\rangle$ for any non-identity $g \in \fG$. 
This implies that the linear functional is maximized at all vertices of $ P $ over all other vertices of $ \fG(P) $. Thus, $P$ is a face of $\fG(P)$. 
\end{proof}

We finish this part by giving formal definitions for hybrid models arising from our constructions, following the discussion in the introduction.

\begin{defn}\label{defn:hybrid}
	Suppose $\u \in \Psi^\circ$ and $Q$ is a polytope contained in an affine space $U'$. If there exists an invertible affine transformation $T$ from $U'$ to $U$ such that $P:=T(Q)$ is $\Phi$-placed and $\Psi$-appropriate, and $\u \in P$, we say $\fG(P)$ is a \emph{hybrid} of the $\fG$-permutohedron $\fG(\u)$ and $Q$.

  Correspondingly, we say a poset $\cF$ is a \emph{hybrid} of the face poset of $\fG(\u)$ and the face poset of $Q$ (or equivalently the face poset of $P$), if it is isomorphic to the face poset of $\fG(P)$.
\end{defn}

\begin{rem}\label{rem:hybrid}
  It follows from Example \ref{ex:g-perm} that a poset $\cF$ is a \emph{hybrid} of the poset $\cO_d$ of ordered set partitions of $[d]$ and another poset $\cG$, if $\cG$ can be realized as a $\Phi_{d-1}$-placed and $\Psi_{d-1}$-appropriate polytope $P$ such that $\cF$ is isomorphic to the face poset of $\fG(P)$.
\end{rem}

\section{(Refined) Fundamental fan} \label{sec:fundamental}
By Theorem \ref{thm:fundamental_cones}, the intersections of the normal cones of a $\Phi$-placed and $\Psi$-appropriate polytope $P\subseteq U$ with the fundamental chamber $\Phi$ determine the normal fan of its symmetrization $\fG(P)$. 
This motivates us to define the combinatorial objects that will be studied in this section.

\subsection{Fundamental fan and determination questions}

Let $P$ be a $\Phi$-placed and $\Psi$-appropriate polytope $P$. 
The set of faces of the intersections $\{ \sigma \cap \Phi \ : \ \sigma \in \Sigma(P)\}$ form a (non-complete) fan. 
We call this fan the \emph{fundamental fan} of $P$, and denote it by $ \ffan(P) $. 
For each $i$, we let $\ffan^i(P)$ be the set of codimension $i$ cones in $\ffan(P)$.

Clearly, we have 
\begin{equation}
\label{eq:fundamental_to_complete}
\Sigma(\fG(P)) = \fG(\ffan(P)) \defeq \{ g \tau \ : \  \tau \in \ffan(P), g \in \fG\}.
\end{equation}

  For simplicity, we denote by $ \cZ(P) $ the face poset of the fundamental fan $ \ffan(P) $.

  \begin{ex}\label{ex:running_fan} 
We compute the fundamental fan of the line segment $S$ of Example \ref{ex:running_example}.
We already have verified in Examples \ref{ex:running_placed} and \ref{ex:running_appropriate} that it is $\Phi_3$-placed and $\Psi_3$-appropriate.

The normal fan $\Sigma(S)$ of $S$ is the fan determined by its perpendicular space $S^\vee$ (defined in \eqref{eq:Pvee}) which is a hyperplane in $W_3$.
Notice that the functionals $ (0,0,0,1) $ and $ (0,1,2,2) $ are both constant on $ S $. Thus, $S^\vee$ (which is $2$-dimensional) is spanned by these two functionals.
The fan $ \Sigma(S) $ consists of three cones: $ \sigma_C := S^{\vee} $, and the two connected components of $ W_3 \backslash \sigma_{C} $.
We denote the component containing the vector $ \f_{1} =(0,1,1,1)$ by $ \sigma_R $ and the other component by $ \sigma_L $.
A depiction of $ \Sigma(S) $ can be seen in Figure \ref{fig:normal_segment}.

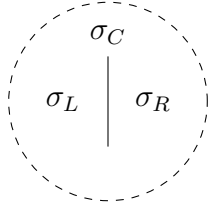
\begin{figure}
\centering
\begin{tikzpicture}[scale=0.6]
\draw (0,1)--(0,-1);
\node[above, ultra thick] at (0,1) {$ \sigma_C $};
\node at (1,0) {$ \sigma_R $};
\node at (-1,0) {$ \sigma_L $};
\draw[dashed] (0,0) circle (2.2 cm);
\node at (0,-3) {};
\end{tikzpicture}
\caption{A visual representation of the normal fan of a segment. The labels of the cones are for Left, Center, and Right.}
\label{fig:normal_segment}
\end{figure}

The fundamental fan $\ffan(S)$ of $S$ is determined by how $S^\vee$ splits the interior of the fundamental chamber $\Phi_3$, which is spanned by $\f_1=(0,1,1,1), \f_2= (0,0,1,1)$ and $\f_3=(0,0,0,1)$. 
The fan $\ffan(S)$ together with its face poset $\cZ(S)$ is depicted in Figure \ref{fig:running_fundamental}. 
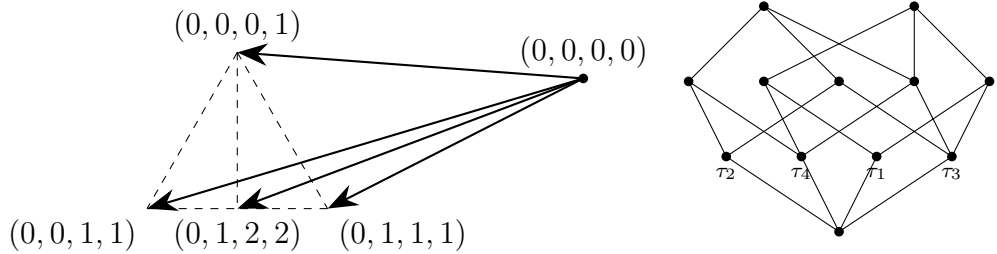
\begin{figure}[ht]
\centering
\begin{tikzpicture}

\begin{scope}[ xshift=-2cm,
x={(300:1 cm)},y={(60:1 cm)}, 
vertex/.style={inner sep=1pt,circle,draw=black,fill=black,thick,anchor=base},
fake/.style={inner sep=1pt,circle,draw=black,thick,anchor=base},
origin/.style={inner sep=1pt,circle,draw=black,thick,anchor=base},
ray/.style={inner sep=1pt,circle,draw=blue,fill=blue,thick,anchor=base},
reference/.style={inner sep=1pt,circle,draw=black,fill=black,thick,anchor=base},
edge/.style={color=purple, thick},
normal/.style={color=blue, dashed},
boundary/.style={color=black, dotted},
scale=0.4
]

\coordinate (012) at (1,1);
\coordinate (021) at (0,1);
\coordinate (102) at (1,0);
\coordinate (120) at (-1,0);
\coordinate (201) at (0,-1);
\coordinate (210) at (-1,-1);
\coordinate (003) at (2,1);
\coordinate (-122) at (1,2);
\coordinate (600) at (-2*1,-2*2);
\coordinate (060) at (-2*1,2*1);
\coordinate (006) at (2*2,2*1);
\coordinate (000) at (10,13);
\coordinate (330) at (-2,-1);
\coordinate (303) at (1,-1);
\coordinate (033) at (1,2);

\node[below left] at (600) {$(0,0,1,1)$};
\node[above] at (060) {$(0,0,0,1)$};
\node[below right] at (006) {$(0,1,1,1)$};
\node[below] at (303) {$(0,1,2,2)$};
\node[above] at (000) {$(0,0,0,0)$};

\node[vertex] at (000) {};
\draw[dashed] (060) -- (303);
\draw[dashed] (600) -- (303);
\draw[dashed] (060) -- (006);
\draw[dashed] (060) -- (600);
\draw[dashed] (006) -- (303);

\draw[-{Stealth[length=4mm, width=3mm]},thick, color=black] (000) -- (060) {};
\draw[-{Stealth[length=4mm, width=3mm]},thick, color=black] (000) --(600) {};
\draw[-{Stealth[length=4mm, width=3mm]},thick, color=black] (000) --(006) {};
\draw[-{Stealth[length=4mm, width=3mm]},thick, color=black] (000) --(303) {};
\end{scope}

\begin{scope}[xshift=4cm, yshift=-1cm,
	vertex/.style={inner sep=1pt,circle,draw=black,fill=black,thick,anchor=base},
	]

	\node[vertex] at (2,0) {};
	\node[vertex] at (0.5,1) {};
	\node[vertex] at (1.5,1) {};
	\node[vertex] at (2.5,1) {};
	\node[vertex] at (3.5,1) {};
	\node[vertex] at (0,2) {};
	\node[vertex] at (1,2) {};
	\node[vertex] at (2,2) {};
	\node[vertex] at (3,2) {};
	\node[vertex] at (4,2) {};
	\node[vertex] at (1,3) {};
	\node[vertex] at (3,3) {};

	\node[below] at (0.5,1) {\tiny $\tau_{2}$};
	\node[below] at (1.5,1) {\tiny $\tau_{4}$};
	\node[below] at (2.5,1) {\tiny $\tau_{1}$};
	\node[below] at (3.5,1) {\tiny $\tau_{3}$};

	\draw (2,0)--(0.5,1) (2,0)--(1.5,1) (2,0)--(2.5,1) (2,0)--(3.5,1);
	\draw (3.5,1)--(4,2)--(2.5,1) (3.5,1)--(3,2)--(1.5,1) (3.5,1)--(2,2)--(0.5,1);
	\draw (0,2)--(1.5,1)--(1,2) (0.5,1)--(0,2) (2.5,1)--(1,2);
	\draw (0,2)--(1,3) (1,2)--(3,3) (2,2)--(1,3) (3,2)--(1,3) (3,2)--(3,3) (4,2)--(3,3);
\end{scope}

\end{tikzpicture}
\caption{The fundamental fan of the line segment of the running example \ref{ex:running_example} together with its face poset.
We have labeled the four rays as $ \tau_{1} = \text{Cone}(\f_1), \tau_{2} = \text{Cone}(\f_2), \tau_{3} = \text{Cone}(\f_3), \tau_{4} = \text{Cone}(\f_2 + \f_1).  $
}
\label{fig:running_fundamental}
\end{figure}
\end{ex}

As shown in Figure \ref{fig:determine}, we know that combinatorics of $\ffan(P)$ ($\fG(P)$, resp.) determines the f-vector of $\ffan(P)$ ($\fG(P)$, resp.). It is natural to ask whether we have determination relation in any of the other arrow directions in the figure. 
Unfortunately, we have the following example. 
\begin{figure}
\begin{tikzcd}[cells={nodes={draw}}, row sep=2cm, column sep=3cm]
    \textup{combinatorics of $\ffan(P)$}
    \arrow[r, smalltext=determines] 
    \arrow[d, smalltext=?]
    \arrow[rd, smalltext=?]
    & \textup{f-vector of $\ffan(P)$}  
    \arrow[d, smalltext=?]
    \\
    \textup{combinatorics of $\fG(P)$} \arrow[r, smalltext=determines] & \textup{f-vector of $\fG(P)$}
\end{tikzcd}
\caption{Determination Questions}
\label{fig:determine}
\end{figure}
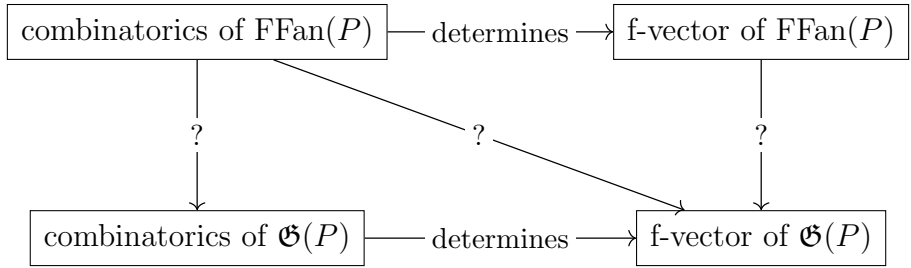

\begin{ex}\label{ex:fvec-counter}
Consider the reflection group $ \fS_5 $ of type $A$.
The two line segments in $ \mathbb{R}^{5} $ defined below are both $\Phi_4$-placed and $\Psi_4$-appropriate:
\[
X := \conv \left\{(2,3,7,8,10), (1,5,6,8,10)\right\}, Y := \conv \left\{(0,3,4,5,10), (0,1,6,7,8)\right\}.
\]
The fundamental fans of both of them are combinatorially isomorphic; they consist of two simplicial 4-cones with a simplicial 3-cone as a common face.
However the $f$-vectors of their respective symmetrizations are
\begin{equation*}
f(\fS_5(X))=(1, 240, 480, 290, 50, 1) \qquad \text{and} \qquad f(\fS_5(Y))=(1, 240, 480, 300, 60, 1).
\end{equation*}
\end{ex}
The above is an counterexample showing that combinatorics of $\ffan(P)$ does \emph{not} determine the f-vector of $\fG(P)$, i.e., the answer to the question mark on the diagonal arrow in Figure \ref{fig:determine} is no. 
Given the existing two determination relation in the figure, this implies that the answer to the questions marks on the two downward arrows are no as well. Given that, we need some extra information in addition to $\ffan(P)$ to determine the combinatorics of $\fG(P)$. 

Although the answers to all the determination questions asked in Figure \ref{fig:determine} are negative, we have a positive answer to a related determination question: $\cZ(P)$ determines $\cF(P)$. This will be the last result of this Section (see Proposition \ref{prop:ffan-det-fp}).

\subsection{Refined Fundamental Fan}
In this part, we will separate cones in $\ffan(P)$ into different sets according to which face of $\Phi$ they ``belong to''. 
Recall that we use $\Sigma(\Phi)$ to denote the fan that is induced by $\Phi$, which is the set of nonempty faces of $\Phi$.

\begin{defn}
  \label{def:carrier}
  Suppose $\tau \in \ffan(P)$ or more generally $\tau$ is convex subset of $\Phi$. %we define $\Phi_\tau \defeq \bigcap_{\tau \subseteq \phi \in \cF(\Phi)} \phi.$ Then 
  Let $\Phi_\tau$ be the inclusion-minimal face of $\Phi$ that contains $\tau$. 
\end{defn} 

\begin{lem}\label{lem:propertyPhi_tau}
  Let $\tau \in \ffan(P)$ or more generally $\tau$ is a convex subset of $\Phi$. Then $\Phi_\tau$ is the unique $\phi \in \Sigma(\Phi)$ satisfying $\tau \subseteq \phi$ and $\tau \cap \phi^\circ \neq \emptyset$. 
  Hence, $\Phi_\tau = \Phi$ if and only if $\tau \cap \Phi^\circ \neq \emptyset$.
\end{lem}
\begin{proof}
  We already know that $\tau \subseteq \Phi_\tau.$
  Suppose $\tau \cap \Phi_\tau^\circ = \emptyset$. 
  Since $\tau \subseteq \Phi_\tau$, we must have that $\tau$ is contained in the boundary of $\Phi_\tau.$
  However, this implies that $\tau$ is contained in one of the facet of $\tau,$ contradicting the definition of $\Phi_\tau$.
  Hence, $\Phi_\tau$ satisfies the two required conditions. 
  Let $\phi$ be a face of $\Phi$ such that $\tau \subseteq \phi$ but $\phi \neq \Phi_\tau$. 
  By the definition of $\Phi_\tau$, we know that $\Phi_\tau$ is a proper face of $\phi$. 
  Thus, $\Phi_\tau$ has an empty intersection with $\phi^\circ$. 
  Therefore, $\tau$, which is contained in $\Phi_\tau$, is disjoint from $\phi^\circ$ as well. 
  Hence, the first conclusion of the lemma follows. 
  The second conclusion follows from the first conclusion together with the fact that $\tau \subseteq \Phi$.
\end{proof}

For each $\phi \in \Sigma(\Phi),$ we define
\begin{equation}
\label{eq:ff_phi}
  \ffan(P, \phi) \defeq \{ \tau \in \ffan(P) \ : \ \Phi_\tau = \phi \}.
\end{equation}

Recall that $f^i(\Sigma)$ is the number of codimension $i$ cones in a fan $\Sigma$. 
Although $\ffan(P, \phi)$ is not necessarily a fan, we extend the operator $f^i$ to them as well. 

Clearly, we obtain a disjoint decomposition of $\ffan(P):$
\begin{equation} 
\label{eq:decomposition}
\ffan(P)= \bigsqcup_{\phi \in \Sigma(\Phi)} \ffan(P, \phi).
\end{equation}
We call $\{ \ffan(P, \phi) \ : \ \phi \in \Sigma(\Phi)\}$ the \emph{refined fundamental fan of $P$}.

\begin{ex} 
  \label{ex:running_fan_decomposed}
  Consider the fundamental fan $\ffan(S)$ in Example \ref{ex:running_fan}. 
	Following the notation from Example \ref{ex:running_fan} we have that the two full-dimensional cones of  $\ffan(S)$ are $ \sigma_R \cap \Phi_3 $ and $ \sigma_L \cap \Phi_3 $.
  The elements in $\Sigma(\Phi_3)$ are described in Definition \ref{defn:ordered_set_partitions} and depicted in Figure \ref{fig:labels}.
  We have 
  \begin{align*}
					\ffan(S, \sigma_{(1,2,3,4)}) =& \{  \sigma_R \cap \Phi_3 ,  \sigma_L \cap \Phi_3 ,  \sigma_C \cap \Phi_3 \},\quad \ffan(S, \sigma_{(12,3,4)}) = \{ \sigma_{(12,3,4)} \}, \\
					\ffan(S, \sigma_{(1,2,34)}) =& \{ \cone(\f_2, \f_2+\f_1), \cone(\f_1, \f_2+\f_1), \cone(\f_2+\f_1) \}, \\
					\ffan(S, \sigma_{(1,23,4)}) =& \{ \sigma_{(1,23,4)} \}. %\quad \ffan(P, r_i) = \{ r_i \}, \quad i=1,2,3, \quad \ffan(P, \0) = \{ \0\}.
    \end{align*}
\end{ex}

For each $\phi \in \Sigma(\Phi),$ we define $\cZ(P, \phi)$ to be the poset on $\ffan(P, \phi)$ ordered by inclusion. 
Equivalently, $\cZ(P, \phi)$ is the subposet induced by $\cZ(P)$ on $\ffan(P, \phi)$.

It is not hard to construct polytopes $P$ and $Q$ such that $\cZ(P)$ is isomorphic to $\cZ(Q)$ but $\cZ(P, \phi)$ is not isomorphic to $\cZ(Q, \phi)$ for some $\phi \in \Sigma(\Phi)$. Therefore, in general, $\cZ(P)$ (the combinatorics of $\ffan(P)$) does not determine every $\cZ(P, \phi)$ (the combinatorics of $\ffan(P,\phi)$).  
However, $\cZ(P)$ does determine $\cZ(P, \Phi)$ as we will show in Proposition \ref{prop:Zdetermine} below. 
We first derive a result characterizing the maximal elements of $\cZ(P, \phi)$ for any $\phi \in \Sigma(\Phi)$. 

\begin{lem} \label{lem:maximalelts}
For each $\phi \in \Sigma(\Phi),$ let $\ffan(P)|_\phi := \{ \tau \in \ffan(P) \ : \ \tau \subseteq \phi\}$ be the \emph{restriction of $\ffan(P)$ to $\phi$}.  
Then 
\begin{equation} 
\label{eq:decomposition_restrict}
\ffan(P)|_\phi = \bigsqcup_{\varphi: \text{ a face of $\phi$}} \ffan(P, \varphi). 
\end{equation}
Moreover, the maximal dimensional cones in $\ffan(P)|_\phi$ have the same dimension as $\phi$ and belong to $\ffan(P, \phi)$. Hence, they are the maximal elements of the poset $\cZ(P, \phi).$  
\end{lem}

\begin{proof} It follows the definition of $\Phi_\tau$ that $\Phi_\tau \subseteq \phi$ if and only if $\tau \subseteq \phi.$ Then Equation \eqref{eq:decomposition_restrict} follows.

  Next, one sees that $\ffan(P)|_\phi$ is a fan supported on $\phi$. Thus, the union of the maximal dimensional cones in $\ffan(P)|_\phi$ is equal to $\phi$. Hence, every maximal cone $\tau$ in $\ffan(P)|_\phi$ must satisfy that $\dim(\tau) = \dim(\phi)$, which implies that $\phi$ is the inclusion-minimal face of $\Phi$ that contains $\tau,$ so $\tau \in \ffan(P,\phi).$ Since $\ffan(P, \phi)$ is a subset of $\ffan(P)|_\phi,$ the maximal cones of $\ffan(P)|_\phi$ are precisely the maximal cones of $\ffan(P,\phi)$. Thus, the last conclusion follows. 
\end{proof}

It is clear that $\ffan(P)|_\phi$ is nonempty and thus has a nonempty set of maximal cones. We have an immediate consequence to Lemma \ref{lem:maximalelts}.  
\begin{cor}\label{cor:nonempty}
  For each $\phi \in \Sigma(\Phi)$, the set $ \ffan(P,\phi) $ (or equivalently, the poset $\cZ(P, \phi)$) is nonempty.
\end{cor}

The following is the main result of this subsection. 
\begin{prop}\label{prop:Zdetermine}
  Suppose $P$ is a $\Phi$-placed and $\Psi$-appropriate polytope. 
  Then we have the following statements giving a characterization for the set $\ffan(P) \setminus \ffan(P,\Phi)$. 
  \begin{enumerate}[label=(B\arabic*)]
    \item\label{itm:B1} If $\tau$ is a codimensional $1$ cone in $\ffan(P)$, then $\tau \not\in \ffan(P, \Phi)$ if and only if $\tau$ is contained in exactly one full-dimensional cone in $\ffan(P)$.
    \item\label{itm:B2} For any $\tau \in \ffan(P)$, we have that $\tau \not\in \ffan(P, \Phi)$ if and only if there exists a codimensional $1$ cone $\tau '\in \ffan(P)$ such that $\tau \subseteq \tau' \not\in \ffan(P,\Phi).$
  \end{enumerate}

  Hence, one can construct the poset $\cZ(P,\Phi)$ from the poset $\cZ(P)$ by removing all elements $z$ in $\cZ(P)$ satisfying one of the followings:
  \begin{enumerate}[label=(b\arabic*)]
    \item $z$ has corank $1$ in $\cZ(P)$ and is covered by exactly one maximal element; 
    \item $z$ is less than one of the elements described in (b1).
  \end{enumerate}

  Therefore, $\cZ(P)$ (the combinatorics of $\ffan(P)$) determines $\cZ(P, \Phi)$ (the combinatorics of $\ffan(P,\Phi)$).
\end{prop}

\begin{proof} 

  By the decompositions  \eqref{eq:decomposition} and \eqref{eq:decomposition_restrict}, we have that 
  \[ \ffan(P) \setminus \ffan(P, \Phi) = \bigcup_{\phi: \text{ a facet of $\Phi$}} \ffan(P)|_\phi.\]
  It follows that $\tau \not\in \ffan(P,\Phi)$ if and only if $\tau \in \ffan(P)|_\phi$ for some facet $\phi$ of $\Phi$. Note that when $\phi$ is a facet of $\Phi$, the maximal dimensional cones in $\ffan(P)|_\phi$ are codimensional $1$ cones in $\ffan(P)$. Hence, \ref{itm:B2} follows. 

  We now prove \ref{itm:B1}. By Lemma \ref{lem:propertyPhi_tau}, we have that $\tau \not\in \ffan(P,\Phi)$ if and only if $\tau \cap \Phi^\circ = \emptyset.$ If $\tau$ is a codimensional $1$ cone in $\ffan(P)$, then $\tau$ is contained in one or two full-dimensional cones in $\ffan(P)$, and $\tau \cap \Phi^\circ =\emptyset$ if and only if $\tau$ is contained in one full-dimensional cone of $\ffan(P)$. Thus, \ref{itm:B1} holds. 

  Clearly, the remaining conclusions of the proposition follow from \ref{itm:B1} and \ref{itm:B2}. 
\end{proof}

\subsection{Intersection interpretation of the refined fundamental fan} 
\label{ssec:omega}

We now give a more detailed description of the cones in $ \ffan(P, \phi) $ as intersection of cones in $ \Sigma(P) $ and $ \Sigma(\Phi) $
For any fan $\Sigma$, we let $\Sigma^\circ =  \{ \sigma^\circ \ : \ \sigma \in \Sigma\}$. 
Note that the (open) cones in $\Sigma^\circ$ gives a disjoint union of the support of $\Sigma.$ 
Since the support of $\Sigma(P)$ is $W$ and the support of $\Sigma(\Phi)$ is $\Phi,$ one sees that cones in 
\[ \{  \sigma^\circ \cap \phi^\circ \neq \emptyset \ : \ \sigma \in \Sigma(P), \phi \in \Sigma(\Phi) \} \]
are disjoint are their union is $\Phi$, which is the support of $\ffan(P)$. 

\begin{defn}\label{defn:OmegaP} Define $\Omega(P) := \{ (\sigma, \phi) \in \Sigma(P) \times \Sigma(\Phi) \ : \ \sigma^\circ \cap \phi^\circ \neq \emptyset \},$
and for each $\phi \in \Sigma(\Phi)$, define $\Omega(P, \phi) := \{ \sigma \in \Sigma(P) \ : \ \sigma^\circ \cap \phi^\circ \neq \emptyset \}.$
\end{defn}

\begin{lem}\label{lem:intersect}
  Let  $(\sigma, \phi)\in \Omega(P)$ and $\tau := \sigma \cap \phi$. 
	Then $\phi=\Phi_\tau$ is the inclusion-minimal face of $\Phi$ containing $\tau$ and $\sigma$ is the inclusion-minimal cone in $\Sigma(P)$ containing $\tau$.
\end{lem}

\begin{proof}
Clearly, $\tau = \sigma \cap \phi$ satisfying $\tau \subseteq \phi$ and $\tau \cap \phi^\circ \neq \emptyset.$ Thus, it follows from Lemma \ref{lem:propertyPhi_tau} that $\phi=\Phi_\tau$. The conclusion about $\sigma$ can be proved analogously. 
\end{proof}

\begin{cor}\label{cor:bij1}
  The map $(\sigma, \phi) \mapsto \sigma \cap \phi$ is a bijection from $\Omega(P)$ to $\ffan(P)$.
Moreover, if $(\sigma, \phi) \in \Omega(P)$ and $\tau = \sigma \cap \phi,$ then $\tau^\circ = \sigma^\circ \cap \phi^\circ$.

Furthermore, for any $(\sigma, \phi), (\sigma', \phi') \in \Omega(P),$ we have
\begin{equation}\label{eq:Sigmaorderiff} \sigma \cap \phi \subseteq \sigma' \cap \phi' \quad \text{if and only if} \quad \sigma \subseteq \sigma' \text{ and } \phi \subseteq \phi'.
\end{equation}
\end{cor}

\begin{proof}
  It follows from Lemma \ref{lem:intersect} that the map is an injection. Next, we have two observations: (1) There are two disjoint unions of $\Phi$ arising from $\Omega(P)$ and $\ffan(P)$:
	\begin{equation}\label{eq:decompositions}
	\bigsqcup_{(\sigma, \phi)\in \Omega(P)} \sigma^\circ \cap \phi^\circ = \Phi = \bigsqcup_{\tau \in \ffan(P)} \tau^\circ.
\end{equation}
  (2) If $\sigma \cap \phi = \tau,$ then $\sigma^\circ \cap \phi^\circ \subseteq \tau^\circ.$ These two observations together imply the second conclusion and the surjectivity of the map. 

  Finally, we prove \eqref{eq:Sigmaorderiff}. The ``if'' direction is obvious. Suppose $\sigma \cap \phi \subseteq \sigma' \cap \phi'$. Let $\tau = \sigma \cap \phi$ and $\tau' =\sigma' \cap \phi'.$ Obviously, we have $\tau \subseteq \tau' \subseteq \sigma'.$ However, by Lemma \ref{lem:intersect}, $\sigma$ is $\sigma$ is the inclusion-minimal cone in $\Sigma(P)$ containing $\tau$. We conclude that $\sigma \subseteq \sigma'.$ We can similarly show that $\phi \subseteq \phi'$.
\end{proof}

We have the following immediately consequence from the above result:
\begin{cor}\label{cor:bij2}
  For each $\phi \in \Sigma(\Phi)$, the map $\sigma \mapsto \sigma \cap \phi$ is a bijection from $\Omega(P, \phi)$ to $\ffan(P, \phi)$.
\end{cor} 

\begin{ex}

Continuing with refined fundamental fan in Example \ref{ex:running_fan_decomposed},
we have every positive dimensional cone of $ \ffan(P) $ expressed as the intersection of a cone in $ \Sigma(P) $ with a cone of $ \Sigma(\Phi_3) $ in Figure \ref{fig:omega}. 
On it, we have each cone from $ \ffan(P)$ expressed as the intersection of a cone from the normal fan of the segment
with a face of the fundamental chamber (with the notation of Figure \ref{fig:labels}). 
From the Figure we are only missing the vertex at the origin which is the intersection $ \sigma_C \cap \{\mathbf{0}\} $.
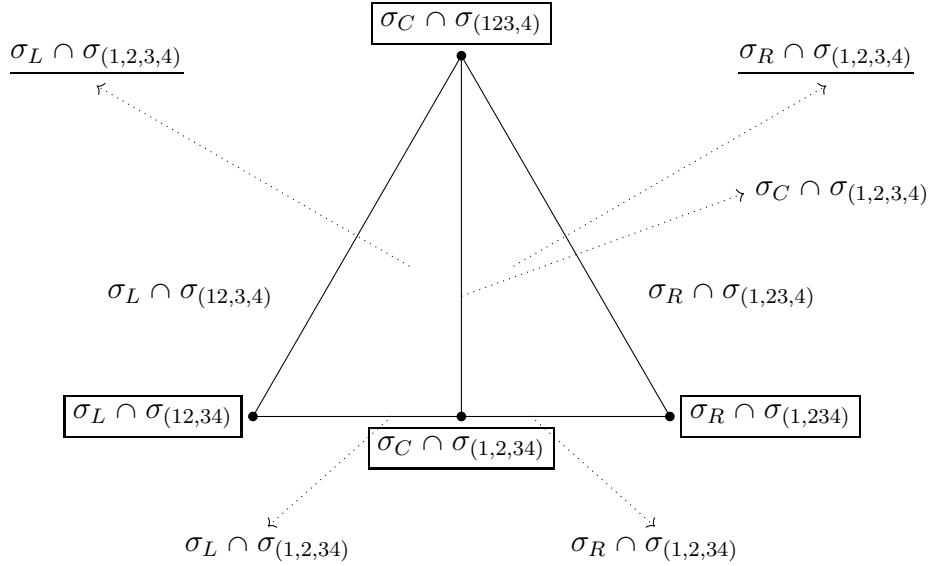
\begin{figure}[ht]
\begin{tikzpicture}[
vertex/.style={inner sep=1pt,circle,draw=black,fill=black,thick,anchor=base},
edge/.style={color=purple, thick},
normal/.style={color=black},
boundary/.style={color=black, dotted},
scale = 0.8
]					

\coordinate (0001) at (90:4);
\coordinate (0011) at (210:4);
\coordinate (0111) at (330:4);
\coordinate (0122) at (270:2);

\node[vertex] at (0001) {};
\node[vertex] at (0011) {};
\node[vertex] at (0111) {};
\node[vertex] at (0122) {};

\draw[normal] (0001)--(0011)--(0111)--cycle;
\draw[normal] (0001)--(0122);

\node at (0:4.5) {$\sigma_R \cap  \sigma_{ (1,23,4)}$};
\node at (0:-4.5) {$\sigma_L \cap   \sigma_{(12,3,4)}$};
\node[above] at (30:7) {$ \underline{\sigma_R \cap \sigma_{(1,2,3,4)}} $};
\node[above] at (150:7) {$ \underline{\sigma_L \cap \sigma_{(1,2,3,4)}} $};
\node[above] at (0001) {$ \boxed{\sigma_C \cap \sigma_{(123,4)}} $};
\node[left] at (0011) {$ \boxed{\sigma_L \cap \sigma_{(12,34)} } $};
\node[right] at (0111) {$ \boxed{\sigma_R \cap \sigma_{(1,234)}} $};
\node[below] at (0122) {$ \boxed{\sigma_C \cap \sigma_{(1,2,34)}} $};
\node[right] at (20:5) {$\sigma_C \cap   \sigma_{(1,2,3,4)}$};
\node[below] at (310:5) {$\sigma_R \cap   \sigma_{(1,2,34)}$};
\node[below] at (230:5) {$\sigma_L \cap   \sigma_{(1,2,34)}$};

\draw[dotted, ->] (300:2.3) to  (310:5);
\draw[dotted, ->] (240:2.3) to  (230:5);
\draw[dotted, ->] (0:0) to  (20:5);
\draw[dotted, ->] (30:1) to (30:7);
\draw[dotted, ->] (150:1) to (150:7);

\end{tikzpicture}
\caption{Each positive dimensional cone in the fundamental fan as an intersection of a normal cone of $ P $ with a face of the fundamental chamber $ \Phi $.}
\label{fig:omega}
\end{figure}
\end{ex}

Another consequence of Corollary \ref{cor:bij1} is a finer decomposition of $ \ffan(P,\phi) $.
\begin{cor}\label{cor:bij3}
  For each $\phi \in \Sigma(\Phi)$ we have the decomposition
	\begin{equation}\label{eq:finer_decomposition}
	\phi^{\circ} = \bigsqcup_{\tau \in \ffan(P,\phi)} \tau^{\circ}
	\end{equation}
\end{cor} 
\begin{proof}
  We consider disjoint unions of $\Phi$ arising from $\Sigma(\Phi)$ and $\Omega(P)$:
\[
  \bigsqcup_{\phi \in \Sigma(\Phi)} \phi^\circ = \Phi = \bigsqcup_{(\sigma, \phi)\in \Omega(P)} \sigma^\circ \cap \phi^\circ = \bigsqcup_{\phi \in \Sigma(\Phi)} \bigsqcup_{\sigma\in \Omega(P, \phi)} \sigma^\circ \cap \phi^\circ. 
\]
Clearly, for each $\phi \in\Sigma(\Phi)$, we must have that $\phi^\circ \supseteq \bigsqcup_{\sigma\in \Omega(P, \phi)} \sigma^\circ \cap \phi^\circ.$ Given that the two sides of above equation are both disjoint union of $\Phi,$ we conclude that
\[ \phi^\circ = \bigsqcup_{\sigma\in \Omega(P, \phi)} \sigma^\circ \cap \phi^\circ.\]
Then \eqref{eq:finer_decomposition} follows from Corollary \ref{cor:bij1}.
\end{proof}

Recall that $ \cZ(P, \phi) $ is the poset on $ \ffan(P, \phi) $ ordered by inclusion.
Given the decomposition presented in Corollary \ref{cor:bij3}, it is natural to ask whether $\cZ(P, \phi)$ has any interesting properties. 
In fact, we have the following result:

\begin{prop}\label{prop:ZPPhi}
  Suppose $P$ is a $\Phi$-placed and $\Psi$-appropriate polytope.
  Then the map $\sigma \mapsto \sigma \cap \Phi$ gives a codimension-preserving bijection from $\Sigma(P)$ to $\ffan(P,\Phi)$. %$\iffan(P)$.
  Therefore, $\cZ(P, \Phi)$ is isomorphic to the dual of $\cF(P)$.
\end{prop}

We need the following auxilliary lemma to prove the above proposition. 
\begin{lem} 
\label{lem:aux}
If $ \gamma_{1}, \gamma_{2} $ are two polyhedral cones with $ \gamma_{1}\cap \gamma_{2}^{\circ} \neq \emptyset $ and $ \gamma_{2} $ full-dimensional,
then we have that $  \gamma_{1}^{\circ} \cap \gamma_{2}^{\circ} \neq \emptyset  $.
\end{lem}

\begin{proof}
Let $ p \in  \gamma_{1}\cap \gamma_{2}^{\circ}  $ and let's assume that $ p \notin \gamma_{1}^{\circ} $.
Then $ p \in \partial \gamma_{1} $.
Consider a small open ball $ B $ centered at $ p $ and contained in $ \gamma_{2}^{\circ} $.
But then $ B \cap \gamma_{1}^{\circ} \neq \emptyset $ and choosing any point $ q $ in the intersection finishes the proof.
\end{proof}

We remark that the conclusion of Lemma \ref{lem:aux} is false without $ \gamma_{2} $ being full-dimensional as the cones 
$ \gamma_{1} =  \text{Cone} \{ (1,1,0), (0,0,1) \}$ and $  \gamma_{2} = \text{Cone} \{ (1,0,0), (0,1,0) \}  $ in $ \mathbb{R}^{3} $ show.

\begin{proof}[Proof of Proposition \ref{prop:ZPPhi}]
We claim that $ \Omega(P, \Phi) = \Sigma(P) $ or equivalently that the relative interior of every normal cone of $ P $ intersects $ \Phi^{\circ} $.
Since $ P $ is $ \Phi $-placed (Definition \ref{def:placed}) we have that $ P^{\vee} \cap \Phi^{\circ} \neq \emptyset $.
For every $ \sigma \in \Sigma(P) $ we have that $ P^\vee \subset \partial \sigma $, hence $ \sigma \cap \Phi^{\circ} \neq \emptyset $.
Then the claim follows Lemma \ref{lem:aux}, and the conclusions of this lemma follow from the claim and Corollary \ref{cor:bij2}.
\end{proof}

Intrigued by Proposition \ref{prop:ZPPhi}, we looked at some small examples we have, and asked the following question:

\begin{ques}
	Is it true that $ \cZ(P,\phi) $ is isomorphic to (the dual of) face poset of a face of $ P $ for all $ \phi \in \Sigma(\Phi) $?
\end{ques}

We finish this section by stating a consequence of Propositions \ref{prop:Zdetermine} and \ref{prop:ZPPhi}.

\begin{prop}\label{prop:ffan-det-fp}
  Suppose $P$ is a $\Phi$-placed and $\Psi$-appropriate polytope. Then $\cZ(P)$ (the combinatorics of $\ffan(P)$) determines $\cF(P)$ (the combinatorics of $P$).
\end{prop}

\section{Symmetrizing posets}
\label{sec:symm-poset}

In the last section, we discussed that given a $\Phi$-placed and $\Psi$-appropriate polytope $P$, the poset $\cZ(P)$ (which captures the combinatorics of the fundamental fan $\ffan(P)$ of $P$) does not determine the face poset $\cF(\fG(P))$ of $\fG(P)$ (which captures the combinatorics of $\fG(P)$). As a result, we introduced the refined fundamental fan of $P$. 
In this section, we will introduce a special kind of poset denoted by $\cR(P)$ to capture the combinatorics of the refined fundamental fan of $P$, and show in Theorem \ref{thm:symposet} that the symmetrization of $\cR(P)$ recovers $\cF(\fG(P))$. 	
	We prove this result by analyzing the symmetrization process on each face of the fundamental chamber.
  We finish this section with a discussion on the $f$-vector of $\fG(P)$.

\subsection{\texorpdfstring{$\Sigma$}{Sigma}-posets and symmetrization}
One way to capture the combinatorics of the refined fundamental fan is to associate each $\tau$ in the face poset $\cZ(P)$ of the fundamental fan $\ffan(P)$ with the cone $\phi=\Phi_\tau$. 
We will make this idea more formal. 

\begin{defn}\label{defn:Sigma-poset}

  Let $\Sigma$ be a fan. 
  We say a poset $\cR$ is a \emph{$\Sigma$-poset} if the underlying set is a subset of $\cF \times \Sigma$ for some poset $\cF$. (Note that if $\cF$ is an antichain, then $\cF$ is just a set without ordering.) 
  Furthermore, we say $\cR$ is \emph{order-consistent}, if for every $(\tau_1, \sigma_1), (\tau_2, \sigma_2) \in \cR$ with $\tau_1 \le \tau_2$, we must have $\sigma_1 \le \sigma_2$.

	We say two $\Sigma$-posets $\cR$ and $\cR'$ are \emph{isomorphic as $\Sigma$-posets} if there is a poset isomorphism $f$ from $\cR$ to $\cR'$ that preserves the second entries, that is for every $(s,\sigma) \in \cR$, if $f(\tau,\sigma)=(\tau',\sigma')$ then $\sigma=\sigma'$. 

	Given a poset $\cF$, and a map $\lambda: \cF \to \Sigma$, we define $\cF^\lambda$ to be the $ \Sigma $-poset on $\{ \tau, \lambda(\tau) \}_{\tau \in \cF}$ ordered by $(\tau, \lambda(\tau)) \le (\tau', \lambda(\tau'))$ if and only if $\tau \le_{\cF} \tau'$. %For each $\sigma \in \Sigma$, we define $\cF|_{\lambda^{-1}(\sigma)}$ to be the induced subposet of $\cF$ on $\lambda^{-1}(\sigma)$.
	\end{defn}
It is clear that $\cF^\lambda$ is a $\Sigma$-poset that is isomorphic to $\cF$ as posets. %Conversely, any $\Sigma$-poset can be constructed as $\cF^\lambda$ for some regular poset $\cF$ and a map $\lambda: \cF \to \Sigma$. Therefore, we often just write a $\Sigma$-poset as $\cF^\lambda$. 
Moreover, $\cF^\lambda$ is order-consistent if and only if $\lambda$ is order-preserving. 
The combinatorics of the refined fundamental fan of $ P $ can be expressed via a $ \Sigma(\Phi) $-poset via the map $\tau \mapsto \Phi_\tau$. %Recall that $ \cZ(P) $ is the face poset of the fundamental fan $ \ffan(P) $

\begin{defn}\label{defn:carr_map} 
  Let $\carr$ be the map from all convex subsets of $\Phi$ to $\Sigma(\Phi)$ that maps $\tau$ to $\Phi_\tau.$ We call $\carr$ the \emph{carrier map} for $\Phi$.

For convenience, we abuse notation and use $\carr$ to denote the restriction of $\carr$ to any set of convex subsets of $\Phi$, e.g., the restriction of $\carr$ to a fan supported on $\Phi$.  
\end{defn}

\begin{defn}
  \label{def:refined_phi_fan}
Let $ P $ be a $\Phi$-placed and $\Psi$-appropriate polytope.
We define the \emph{face poset of the refined fundamental fan of $P$}, denoted by $\cR(P)$, to be the $\Sigma(\Phi)$-poset $\cZ(P)^\carr$.

\end{defn}
Hence, $\cR(P)$ is isomorphic to $\cZ(P)$ as posets, but it contains more information encoded in the carrier map $\carr$.

  It is not true that every $\Sigma$-poset can be written as $\cF^\lambda$. 
  However, almost all the $\Sigma(\Phi)$-posets considered in this paper can be expressed as $\cF^\carr$. %Moreover, the map $\lambda$ is always the one that maps $\tau$ to $\Phi_\tau$. 
  Below we include a basic property of $\Sigma(\Phi)$-posets of this form. 
  \begin{lem}\label{lem:orderconsist}
  Let $\cF$ be a poset on a set of convex subsets of $\Phi$ ordered by inclusion. Then $\carr$ is an order-preserving map from $\cF$ to $\Sigma(\Phi)$. Hence, $\cF^\carr$ is order-consistent. 

  In particular, for a $\Phi$-placed and $\Psi$-appropriate polytope $P$, the face poset $\cR(P)$ of its refined fundamental fan is order-consistent. 
\end{lem}

\begin{proof}
  The conclusion that $\carr$ is order-preserving following directly from the definition of $\carr(\tau)$ is the inclusion-minimal face of $\Phi$ that contains $\tau$.
\end{proof}

We now define the $ \fS $-symmetrization of an arbitrary $ \Sigma(\Phi) $-poset.

\begin{defn} \label{defn:symmetrize-poset}
      Let $\cR$ be a $\Sigma(\Phi)$-poset. The $ \fG $-symmetrization of $ \cR $ is the $\Sigma(\cH)$-poset $ \fG\cR $ whose underlying set is
      \[\{ (\tau,  \eta)  \ : \ (\tau, g^{-1}\eta) \in \cR, \text{ for some $g \in \fG$} \},\] 
	and with partial order defined by $(\tau_1, \eta) \le (\tau_2, \eta)$ if and only if there exists $g \in \fG$ such that
	$(\tau_1, g^{-1} \eta) \le_\cR (\tau_2, g^{-1} \eta)$.
\end{defn}

It is clear that if $\cR$ and $\cR'$ are two isomorphic $\Sigma(\Phi)$-posets. Then their $\fG$-symmetrizations are isomorphic as posets or as $\Sigma(\cH)$-posets. 

\begin{thm}\label{thm:symposet}
  Let $ P $ be a $\Phi$-placed and $\Psi$-appropriate polytope. 
  Then the $ \fG $-symmetrization $\fG\cR(P)$ of $ \cR(P) $ is isomorphic to the face poset of $\Sigma(\fG(P)).$ 
	Therefore, the face poset of $\fG(P)$ is isomorphic to the dual of $\fG\cR(P)$.
\end{thm}

The following lemma is a more explicit description of $\fG\cR(P)$.
It follows directly from the definitions.

\begin{lem}\label{lem:1stonGR}
The $ \fG $-symmetrization $\fG\cR(P)$ of $ \cR(P) $ is the poset whose underlying set is
\[\{ (\tau,  g \kappa(\tau))  \ : \ \tau \in \ffan(P), g \in \fG \},\] 
and with partial order defined by $(\tau_1, g_1 \kappa(\tau_1)) \le (\tau_2, g_2 \kappa(\tau_2))$ if and only if $\tau_1 \subseteq \tau_2$ and there exists $g \in \fG$ such that
	\begin{equation}
	  g^{-1} g_1 \kappa(\tau_1) = \kappa(\tau_1) \text{ and } g^{-1} g_2 \kappa(\tau_2) = \kappa(\tau_2). 
	  \label{eq:Phi_tau_stab}
	\end{equation}
\end{lem}

We need to develop more results before proving Theorem \ref{thm:symposet}. 
Before that, we finish this subsection with a result that connecting $\cR(P)$ with $\Omega(P)$ disscussed in the last section.

\begin{lem}\label{lem:OmegaPiso}
  We can consider $\Omega(P)$ as a $\Sigma(\Phi)$-poset by defining 
  \begin{equation}\label{eq:OmegaPorder} (\sigma, \phi) \le (\sigma', \phi') \text{ in $\Omega(P)$ \quad if and only if \quad} \sigma \subseteq \sigma' \text{ and } \phi \subseteq \phi'.
\end{equation}
The map $(\sigma, \phi) \mapsto (\sigma \cap \phi, \phi)$ gives a $\Sigma(\Phi)$-poset isomorphism from $\Omega(P)$ to $\cR(P)$. 
\end{lem}
\begin{proof}
  It follows directly from Lemma \ref{lem:intersect}, Corollaries \ref{cor:bij1} and \ref{cor:bij2}.
\end{proof}

\begin{rem}
  The poset $\Omega(P)$ considered in the above lemma is an example of $\Sigma(\Phi)$-poset that cannot be expressed as $\cF^\lambda$ in general. 
	Indeed, for most of examples we consider such as Example \ref{ex:running_fan_decomposed}, the first entries $\sigma$ appeaing in elements of $\Omega(P)$ are \emph{not} distinct, and hence do not form a poset. 
	Although it is sometimes handy to describe $\cR(P)$ using $\Omega(P)$, for the discussion of symmetrization of posets, the representation of $\cR(P)$ seems to be the right choice. 
\end{rem} 

\subsection{ Refined analysis } 
In order to prove Theorem \ref{thm:symposet}, we first will rewrite \eqref{eq:fundamental_to_complete} as a disjoint sum of elements in order to understand which cones appearing in $\Sigma(\fG(P))$ better. 
We start with reviewing some group-theoretic concepts and a result. 

\begin{defn}\label{defn:stb}
  Let $\sigma$ be a cone (or a subset) in $W$. 
	We say $g \in \fG$ \emph{stabilizes} $\sigma$ if $g \sigma = \sigma.$ 
	The \emph{stabilizer of of $\sigma$ (with respect to $\fG$)} is $\fG_\sigma:= \{ g \in \fG \ : \ g \sigma = \sigma \}$, which is the subgroup of $\fG$ that consists all elements that stabilize $\sigma.$ 
  The \emph{orbit of $\sigma$ (in $\fG$)} is $\fG \sigma := \{ g \sigma \ : \ g \in \fG\}$. 
\end{defn}

It is a basic group theory result that there is a bijection between left cosets of the stabilizer $\fG_\tau$ and the orbit $\fG \cdot \tau$. More precisely: 
\begin{lem}\label{lem:sta-orb}
  Suppose $\{g_1, \dots, g_\ell\}$ is a set of representatives of left cosets of $\fG_\tau,$ that is, $\ds \fG = \bigsqcup_{i=1}^\ell g_i \fG_\tau$. Then $\fG \tau$ consists of $g_1 \tau, \dots, g_\ell \tau$.
Therefore, $|\fG_\tau| \cdot |\fG \tau| = |\fG|.$ \end{lem}

The following lemma summarizes results from \cite[Secion 1.12]{humphreys} that will be useful for the discussion in this part. 
\begin{lem}\label{lem:stabilizer} The followings are true:
  \begin{enumerate}
    \item\label{itm:unique} 
  For any $\z \in W$, there is a unique $\w \in \Phi$ such that $\z \in \fG \w$. Hence,
  $\ds W = \bigsqcup_{\w \in \Phi} \fG \w.$
  
\item\label{itm:stbunique} For any $\eta\in \Sigma(\cH)$, there exists a unique subgroup of $\fG$ that is the stabilizer of every point in $\eta^\circ.$  
  Furthermore, if $\eta= \Phi$, this unique stabilizer is the trivial group of $\fG$.

\item\label{itm:stbinclusion} Let $\eta, \eta' \in \Sigma(\cH),$ $\w \in \eta^\circ$ and $\w' \in (\eta')^\circ$. If $\eta' \subseteq \eta$, then $\fG_{\w} \subseteq \fG_{\w'}.$
  \end{enumerate}
\end{lem}

The following corollary is an immediate consequence of the above lemma.  
\begin{cor}\label{cor:stabilizer} Let $\tau, \tau' \subseteq \Phi$ and $\phi \in \Sigma(\Phi)$.  Then the followings are true:
    \begin{enumerate}
      \item\label{itm:distinctorbits} %Let $\tau, \tau' \subseteq \Phi$. 
	If $\tau \neq \tau'$, then $\fG \tau$ and $\fG \tau'$ are disjoint. 
      \item\label{itm:pointwise} %Let $\tau \subseteq \Phi$. 
	If $g \in \fG$ stabilizes $\tau,$ then $g$ stabilizes every point of $\tau.$ Hence, $\ds \fG_\tau = \cap_{\w \in \tau} \fG_\w$ is a pointwise stabilizer of $\tau$.

      \item\label{itm:connection} 
	Let $\w \in \phi^\circ$. 
	If $\tau \subseteq \phi$ and $\tau \cap \phi^\circ \neq \emptyset,$ then $\fG_\tau=\fG_\w$ is the unique stabilizer for all points in $\phi^\circ$. 
	Consequently, we have $\fG_\phi= \fG_\w = \fG_\tau$. 
      
      \item\label{itm:Phitau} Suppose $\phi=\kappa(\tau)$. Then $\fG_\phi=\fG_\tau.$ Thus, for any $g, g' \in \fG$, we have
	\[ g \phi= g'\phi \text{ if and only if } g \tau = g' \tau.\]
    \end{enumerate}

  \end{cor}
  We remark that the above corollary holds if we replace $\Phi$ with any chamber of $\Sigma(\cH)$.

  \begin{proof} Conclusions \eqref{itm:distinctorbits} and \eqref{itm:pointwise} follow from Lemma \ref{lem:stabilizer}/\eqref{itm:unique}. 
    We now prove \eqref{itm:connection}.
    The conclusion that $\fG_\tau = \fG_\w$ follows from \eqref{itm:pointwise}, Lemma \ref{lem:stabilizer}/\eqref{itm:stbunique}\eqref{itm:stbinclusion}, and the fact that for any point $\w' \in \phi,$ there exists a (unique) face $\phi'$ of $\phi$ such that $\w' \in (\phi')^\circ.$ Finally, one notices that $\phi$ itself is a $\tau$ satisfying $\tau \cap \phi^\circ \neq \emptyset.$ Thus, $\fG_\phi= \fG_\w$.

    Finally, \eqref{itm:Phitau} follows from \eqref{itm:connection} and Lemma \ref{lem:propertyPhi_tau}. 
  \end{proof}

\begin{rem}\label{rem:parabolic}
  The groups that arise as the stabilizers $\fG_\phi$ of faces $\phi$ of $ \Phi $ are the so-called \emph{parabolic subgroups} of the reflection group $ \fG $ \cite[Chapter 1.10]{humphreys}.
  They are in bijection with faces of $ \Phi $ \cite[Chapter 1.15]{humphreys}, and are all different, i.e., $\fG_{\phi} \neq \fG_{\phi'}$ if $\phi\neq \phi' \in \Sigma(\Phi)$. 
\end{rem}

\begin{proof}[Proof of Theorem \ref{thm:symposet}]
  Using Corollary \ref{cor:stabilizer}/\eqref{itm:distinctorbits}, we can rewrite \eqref{eq:fundamental_to_complete} as
  \begin{equation}
  	\label{eq:fundamental_to_complete1}
  	\Sigma(\fG(P)) = \fG(\ffan(P)) = \bigsqcup_{\tau \in \ffan(P)} \fG \tau = \bigsqcup_{\phi \in \Sigma(\Phi)} \bigsqcup_{\tau \in \ffan(P, \phi)} \fG \tau. 
  \end{equation}
  We define a map $f: \Sigma(\fG(P)) \to  \fG\cR(P) $ in the following way: for each $\tau_0 \in \Sigma(\fG(P))$, by the above decomposition, there exists a unique $\tau \in \ffan(P)$ such that $\tau_0 = g \tau$ for some $g \in \fG$. 
	We define
  \[ f(\tau_0) = (\tau, g \kappa(\tau)).\]
	Note that by Corollary \ref{cor:stabilizer}/\eqref{itm:Phitau}, the definition of $f(\tau_0)$ is independent from the choice of $g$ as long as it satisfies $\tau_0 = g \tau$, and hence $ f $ is well-defined. For the same reason, given $(\tau, g \kappa(\tau)) \in \fG\cR(P)$, there exists a unique $t_0 =gt$ as its preimage. Hence, $f$ is a bijection. 

  It remains to prove that $g_1 \tau_1 \subseteq g_2 \tau_2$, as faces of $ \Sigma(\fG(P))$, if and only if $(\tau_1, g_1 \kappa(\tau_1)) \le (\tau_2, g_2 \kappa(\tau_2))$ in the poset $\fG\cR(P)$. 
	However, by Lemma \ref{lem:1stonGR} and Corollary \ref{cor:stabilizer}/\eqref{itm:Phitau}, one sees that the latter is equivalent to  
  \begin{equation}
  \label{eq:equiv-cond}
\tau_1 \subseteq \tau_2 \text{ and } \exists g\in \fG \text{ such that  $g^{-1} g_i \in \fG_{\tau_i}$ for $i=1$ and $2.$}
\end{equation}
By the Corollary \ref{cor:stabilizer}/\eqref{itm:distinctorbits}, we must have $g^{-1} g_i \tau_i = \tau_i$ for $i=1$ and $2$.
It is then clear that \eqref{eq:equiv-cond} implies that $g_1 \tau_1 = g\tau_1 \subseteq g\tau_2 = g_2 \tau_2$. 
Now suppose $g_1\tau_1 \subseteq g_2\tau_2$. 
Then there exists $g \in \fG$ such that $g_1\tau_1 \subseteq g_2\tau_2 \subseteq g\Phi.$ Thus, $g^{-1} g_1 \tau_1 \subseteq g^{-1} g_2 \tau_2 \subseteq \Phi.$ 
 Hence, \eqref{eq:equiv-cond} follows, completing the proof.
\end{proof}

We end this part with an example.
In order to illustrate the complete poset of the symmetrization, we will use an example in two dimensions, as opposed to our running example in three dimensions.

\begin{ex}
Consider the line segment $ L = \conv \{ (1,5,6), (2,3,7) \} \subset \mathbb{R}^{3} $.
Since both vertices have the same sum of coordinates (namely 12) we have $ L \subset U_{2}^{(1,5,6)} $.
Thus we can consider its $ \fS_2 $-symmetrization. 

Note that $ L $ is $ \Phi_2 $-placed and $ \Psi_{2} $-appropriate.
The two normal cones to the the vertices split the 2-dimensional fundamental chamber $ \Phi_{2} $ in two.
Thus the fundamental fan $ \ffan_{W_2}(L) $ consists of two 2-dimensional cones intersecting in a ray, see Figure \ref{fig:line_A} in the middle.
The face poset of the fundamental fan is depicted in the right of Figure \ref{fig:line_A} with the labels on top of the poset.
To be more precise let $ \tau_1 $ be the ray spanned by $ \f_1^{(3)}=(0,1,1) $ $ W_2 $,
and similarly let $ \tau_{2} $ be the ray spanned by $ \f_2^{(3)} = (0,0,1)$.
Furthermore, the intersection of the two 2-dimensional cones in the fundamental fan is the ray spanned by $ (0,1,2)$, 
which is denoted by $ c $ in Figure \ref{fig:line_A}.

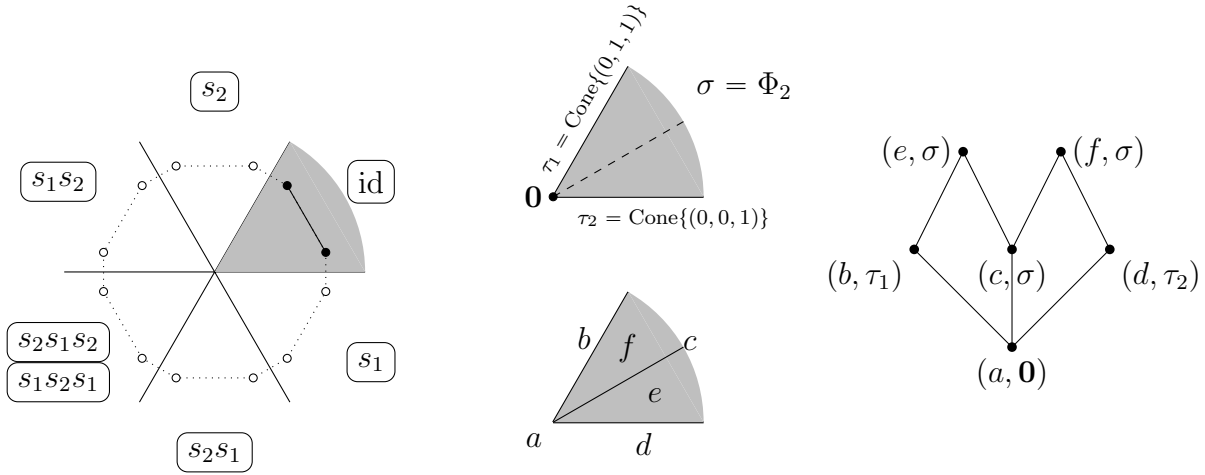
\begin{figure}[ht]
\begin{tikzpicture}
\begin{scope}[xshift=-6cm]

\draw[fill=white!75!black, draw=none] (0:2cm) arc (0:60:2);
\draw (240:2cm)--(60:2cm) (180:2cm)--(0:2cm) (120:2cm)--(300:2cm);
\draw[fill=white!75!black, draw=none](0:0cm) -- (0:2cm) -- (60:2cm) -- cycle;
\filldraw[black] (10:1.5cm) circle (1.5 pt);
\filldraw[black] (50:1.5cm) circle (1.5 pt);
\draw (10:1.5cm) -- (50:1.5cm);

\node[draw, rounded corners] at (30:2.4cm) {id};
\node[draw, rounded corners] at (90:2.4cm) {$ s_{2} $};
\node[draw, rounded corners] at (150:2.4cm) {$ s_{1}s_{2} $};
\node[above, draw, rounded corners] at (210:2.4cm) {$ s_{2}s_{1}s_{2} $};
\node[below, draw, rounded corners] at (210:2.4cm) {$ s_{1}s_{2}s_{1} $};
\node[draw, rounded corners] at (270:2.4cm) {$ s_{2}s_{1} $};
\node[draw, rounded corners] at (330:2.4cm) {$ s_{1} $};

\draw[dotted] (10:1.5cm) -- (50:1.5cm) -- (70:1.5cm) -- (110:1.5cm) -- (130:1.5cm) -- (170:1.5cm)--(190:1.5cm) -- (230:1.5cm) -- (250:1.5cm) -- (290:1.5cm) -- (310:1.5cm) -- (350:1.5cm)-- cycle;

\draw[fill=white] (70:1.5cm) circle (1.5 pt);
\draw[fill=white] (170:1.5cm) circle (1.5 pt);
\draw[fill=white] (-170:1.5cm) circle (1.5 pt);
\draw[fill=white] (290:1.5cm) circle (1.5 pt);
\draw[fill=white] (-50:1.5cm) circle (1.5 pt);

\draw[fill=white] (-10:1.5cm) circle (1.5 pt);
\draw[fill=white] (110:1.5cm) circle (1.5 pt);
\draw[fill=white] (130:1.5cm) circle (1.5 pt);
\draw[fill=white] (230:1.5cm) circle (1.5 pt);
\draw[fill=white] (250:1.5cm) circle (1.5 pt);

\end{scope}

\begin{scope}[xshift=-1.5cm, yshift=1cm]
\draw[thick] (240:0cm)--(60:2cm) (180:0cm)--(0:2cm) ;
\draw[fill=white!75!black, draw=none] (0:2cm) arc (0:60:2);
\draw[fill=white!75!black, draw=none](0:0cm) -- (0:2cm) -- (60:2cm) -- cycle;

\draw[dashed] (0:0)--(30:2cm);
\node[right] at (40:2.3cm) {$\sigma = \Phi_2$};
\node[below] at (0:1.6cm) {\tiny $\tau_{2} = \text{Cone}\{(0,0,1)\}$};
\node[anchor=north, rotate=60] at (80:1.6cm) {\tiny $ \tau_{1} = \text{Cone}\{(0,1,1)\} $};
\draw[fill=black] (0:0) circle (1.5 pt);
\node[left] at (0:0) {$ \mathbf{0} $};
\end{scope}

\begin{scope}[xshift=-1.5cm, yshift=-2cm]
\draw[thick] (240:0cm)--(60:2cm) (180:0cm)--(0:2cm) ;
\draw[fill=white!75!black, draw=none] (0:2cm) arc (0:60:2);
\draw[fill=white!75!black, draw=none](0:0cm) -- (0:2cm) -- (60:2cm) -- cycle;

\draw (0:0)--(30:2cm);
\node[below] at (0:1.2cm) {$ d $};
\node at (30:2.1cm) {$ c $};
\node at (70:1.2cm) {$ b $};
\node[below left] at (0:0) {$ a $};
\node at (15:1.4cm) {$ e $};
\node at (45:1.4cm) {$ f $};
\end{scope}

\begin{scope}[xshift=2cm, yshift=-1cm,
vertex/.style={inner sep=1pt,circle,draw=black,fill=black,thick,anchor=base},
scale = 1.3]

\node[vertex] at (2,0) {};
\node[vertex] at (1,1) {};
\node[vertex] at (2,1) {};
\node[vertex] at (3,1) {};
\node[vertex] at (1.5,2) {};
\node[vertex] at (2.5,2) {};

\node[below] at (2,0) { $(a, \mathbf{0})$};
\node[below left] at (1,1) { $ (b, \tau_{1})$};
\node[below] at (2,1) { $(c, \sigma)$};
\node[below right] at (3,1) { $(d, \tau_{2})$};
\node[left] at (1.5,2) { $(e, \sigma)$};
\node[right] at (2.5,2) { $(f, \sigma)$};

\draw (2,0) -- (1,1) (2,0) -- (2,1) (2,0) -- (3,1) (1,1) -- (1.5,2) (2,1) -- (1.5,2) (2,1) -- (2.5,2) (3,1) -- (2.5,2);
\end{scope}
\end{tikzpicture}
\caption{The symmetrization of $ L $ on the left.
In the middle we have the fundamental fan $ \ffan_{W_2}(L) $. 
And finally on the right we have the refined fundamental fan.
The names of the cones in the fundamental fan are denoted $ \left\{ a,b,c,d,e,f \right\} $,
while the labels are $ \left\{ \mathbf{0}, \tau_{1}, \tau_{2}, \sigma \right\} $}
\label{fig:line_A}
\end{figure}
The left hand side of Figure \ref{fig:line_A} also has depicted a correspondence between elements of $ \fG_2 $ and the chambers of the arrangement 
(here $ s_{i} $ is the reflection over the line spanned by $ \tau_{i} $).

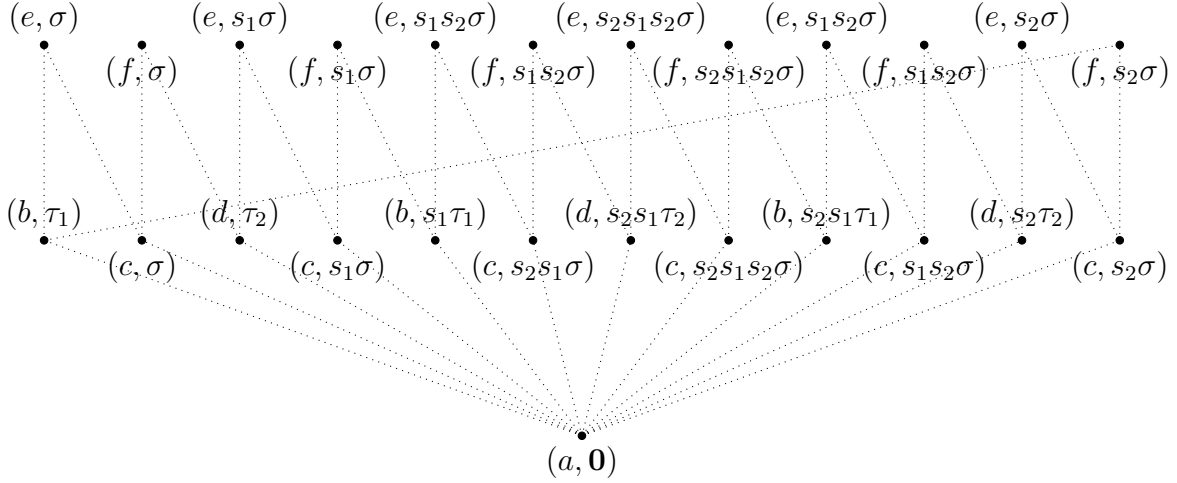
\begin{figure}
\begin{tikzpicture}[ vertex/.style={inner sep=1pt,circle,draw=black,fill=black,anchor=base}, scale=1.3]

	\draw[dotted] (0,2) -- (0,4) (5.5,0)--(0,2) ;
	\draw[dotted] (0,2) -- (11,4);
	\foreach \x in {1,...,11}
	\draw[dotted] (\x,2)--(\x,4) (\x,2)--(\x-1,4) (5.5,0)--(\x,2) ;

	\node[vertex] at (0,2) {};
	\node[above] at (0,2) {$ (b, \tau_{1}) $};
	\node[vertex] at (1,2) {};
	\node[below] at (1,2) {$ (c, \sigma) $};
	\node[vertex] at (2,2) {};
	\node[above] at (2,2) {$ (d, \tau_{2}) $};
	\node[vertex] at (3,2) {};
	\node[below] at (3,2) {$ (c, s_{1}\sigma) $};
	\node[vertex] at (4,2) {};
	\node[above] at (4,2) {$ (b, s_{1}\tau_{1}) $};
	\node[vertex] at (5,2) {};
	\node[below] at (5,2) {$ (c, s_{2}s_{1}\sigma) $};
	\node[vertex] at (6,2) {};
	\node[above] at (6,2) {$ (d, s_{2}s_{1}\tau_{2}) $};
	\node[vertex] at (7,2) {};
	\node[below] at (7,2) {$ (c, s_{2}s_{1}s_{2}\sigma) $};
	\node[vertex] at (8,2) {};
	\node[above] at (8,2) {$ (b, s_{2}s_{1}\tau_{1}) $};
	\node[vertex] at (9,2) {};
	\node[below] at (9,2) {$ (c, s_{1}s_{2}\sigma) $};
	\node[vertex] at (10,2){};
	\node[above] at (10,2){$ (d, s_{2}\tau_{2}) $};
	\node[vertex] at (11,2){};
	\node[below] at (11,2) {$ (c, s_{2}\sigma) $};

	\node[vertex] at (0,4) {};
	\node[above] at (0,4) {$ (e, \sigma) $};
	\node[vertex] at (1,4) {};
	\node[below] at (1,4) {$ (f, \sigma) $};
	\node[vertex] at (2,4) {};
	\node[above] at (2,4) {$ (e, s_{1}\sigma) $};
	\node[vertex] at (3,4) {};
	\node[below] at (3,4) {$ (f, s_{1}\sigma) $};
	\node[vertex] at (4,4) {};
	\node[above] at (4,4) {$ (e, s_{1}s_{2}\sigma) $};
	\node[vertex] at (5,4) {};
	\node[below] at (5,4) {$ (f, s_{1}s_{2}\sigma) $};
	\node[vertex] at (6,4) {};
	\node[above] at (6,4) {$ (e, s_{2}s_{1}s_{2}\sigma) $};
	\node[vertex] at (7,4) {};
	\node[below] at (7,4) {$ (f, s_{2}s_{1}s_{2}\sigma) $};
	\node[vertex] at (8,4) {};
	\node[above] at (8,4) {$ (e, s_{1}s_{2}\sigma) $};
	\node[vertex] at (9,4) {};
	\node[below] at (9,4) {$ (f, s_{1}s_{2}\sigma) $};
	\node[vertex] at (10,4){};
	\node[above] at (10,4){$ (e, s_{2}\sigma) $};
	\node[vertex] at (11,4) {};
	\node[below] at (11,4) {$ (f, s_{2}\sigma) $};

	\node[vertex] at (5.5,0) {};
	\node[below] at (5.5,0) {$ (a, \mathbf{0}) $};

\end{tikzpicture}

\caption{Complete face poset of the normal fan of the $ \fS_2 $-symmetrization of $ L $ as the symmetrization of its refined fundamental fan.}
\label{fig:symmetrized}
\end{figure}

\end{ex}

\subsection{Face vectors of symmetrizations} 
In this part, we derive a formula that computes the $f$-vector of $\fG(P)$ from the combinatorics of the refined fundamental fan, as a consequence of results that have been presented in this section.
At the end we show that even though in general the combinatorics of the fundamental fan $\ffan(P)$ does not determine the $f$-vector of $\fG(P)$ as shown in Example \ref{ex:fvec-counter}, it does determine the first two entries of the $f$-vector of $\fG(P)$ (see Remark \ref{rem:determine}). 

\begin{prop}\label{prop:computefvec}
  Suppose $P$ is a $\Phi$-placed and $\Psi$-appropriate polytope. For $\phi \in \Sigma(\Phi)$, let $m(\phi)=|\fG_\phi|$.  %be the number of $g \in \fG$ such that $\phi \subseteq g\Phi$. 
  Then
  \begin{equation}
f_i(\fG(P)) = \sum_{\text{$\phi:$ face of $\Phi$ of codimension at most $i$}} f^i(\ffan(P,\phi)) \cdot \frac{|\fG|}{m(\phi)}.
    \label{eq:computefvec}
  \end{equation}
\end{prop}

\begin{proof}
  By Lemma \ref{lem:face&fan}, %/\eqref{itm:dual}, 
  $f_i(\fG(P)) = f^i(\Sigma(\fG(P)))$. It follows from \eqref{eq:fundamental_to_complete1} and Lemma \ref{lem:sta-orb} that $f^i(\Sigma(\fG(P)))$ is computed by \eqref{eq:computefvec}.
  
\end{proof}

\begin{ex}
We apply Proposition \ref{prop:computefvec} to Example \ref{ex:running_example} to compute the f-vector of the symmetrization of the line segment $S$.
The refined fundamental fan is described in detail in Example \ref{ex:running_fan_decomposed}.
Following Proposition \ref{prop:computefvec}, we must compute $ m(\phi) $ for each nonempty face $\phi$ of $\Phi_3$.
We use the same notation as in Example \ref{ex:running_fan_decomposed}.
\begin{equation*}
\begin{array}{c|c|c|c|c|c|c|c|c}
\text{Face $\phi$} & \Phi & \phi_1 & \phi_2 & \phi_3 & r_1 & r_2 & r_3 & 0 \\
\hline
m(\phi) & 1 & 2 & 2 & 2 & 6 & 4 & 6 & 24 
\end{array}
\end{equation*}
From here we can compute, via Equation \eqref{eq:computefvec}, that the f-vector of $ P $ is $ (48,72,26,1) $.

\end{ex}

In general, it is not so easy to compute $f^i(\ffan(P,\phi))$ for an arbitrary face $\phi$ of $\Phi.$ 
However, it follows from Proposition \ref{prop:ZPPhi} that if $\phi=\Phi$, these numbers are precisely the $f$-vector of $P$.

\begin{lem}\label{lem:countiffan}
  Suppose $P$ is a $\Phi$-placed and $\Psi$-appropriate polytope.
  Then $f^i(\ffan(P,\Phi)) = f_i(P)$ for each $i$.
\end{lem}

Now we are able to determine two entries of the $f$-vector of $\fG(P)$.
\begin{cor}\label{cor:detfvec}
Suppose $P$ is a $\Phi$-placed and $\Psi$-appropriate polytope. Then the first two entries in the $f$-vector of $\fG(P)$ is determined by $f_0(P), f_1(P)$ and $f^1(\ffan(P))$.
More precisely, 
  \begin{align}
    f_0(\fG(P)) =& |\fG| f_0(P), \label{eq:detv} \\ 
    f_1(\fG(P)) =& \frac{|\fG|}{2} \left (f^1(\ffan(P)) + f_1(P) \right). \label{eq:dete}
    \end{align}
\end{cor}

\begin{proof}
Note that $m(\Phi)=1$ and for any facet $\phi$ of $\Phi$, we always have $m(\phi)=2.$ 
Then the formulas follow from Formula \eqref{eq:computefvec} and Lemma \ref{lem:countiffan} with $i=0$ and $1$.
\end{proof}
\begin{rem}\label{rem:determine} It follows from Proposition \ref{prop:ffan-det-fp} that $\cZ(P)$ (the combinatorics of $\ffan(P)$) determines $f_0(P)$ and $f_1(P)$. Since one clearly can read $f^1(\ffan(P))$ from $\cZ(P)$. Corollary \ref{cor:detfvec} implies that $\cZ(P)$ determines the first two entries in the $f$-vector of $\fG(P)$. 
\end{rem}

The above corollary together with Euler's formula has the following consequence. We leave it to the readers to check the details.
\begin{cor}\label{cor:fvdetfv} 
Suppose $U$ and $W$ are of dimension at most $3,$ and $P$ is a $\Phi$-placed and $\Psi$-appropriate polytope. 
Then $f^0(\ffan(P))$ and $f^1(\ffan(P))$ determine the $f$-vector of $\fG(P)$.
\end{cor}

\section{Realizing a symmetric poset}
\label{sec:application}

In the last three sections, our discussion has been focused on deriving results of the symmetrization of a $\Phi$-placed and $\Psi$-appropriate polytope $P$. 
Thus, the starting point is always a given polytope $P$. 
Our results allow us to form a combinatorics structure to capture the combinatorics of the refined fundamental fan of $P$, and show that it contains enough information for us to construct the face poset of its symmetrization $\fG(P)$. 
It follows from Theorem \ref{thm:symposet} that the face poset of $\fG(P)$ is always $\fG$-symmetric.  
In this section, we turn our attention to a related realization problem:

\begin{ques}
  Given a finite $\fG$-symmetric poset $\cF$, can we construct a polytope $P$ such that $\cF$ is the face poset of its symmetrization $\fG(P)$?
\end{ques}

\begin{rem}
  Following Definition \ref{defn:hybrid}, the above question is equivalent to determining whether a $\fG$-symmetric poset $\cF$ is a hybrid of the face poset of of a $\fG$-permutohedron and another poset $ \mathcal{G} $ that is polytopal. 
	Hence, all the discussion below can be applied to this ``hybrid poset'' version question. 
	To save space, we will leave the details of the translation to the readers. 
\end{rem}

It turns out the theories we developed so far can reduce the realization problem of a symmetric poset to a smaller task of realizing its generating subposet. 
We need one more concept before state the main theorem of this section.

\begin{defn}
	Let $\cT$ be a $\fG$-symmetric poset with a generating subposet $\cZ$. 
	We say $ \cZ $ is \emph{$ \fG $-compatible} if for every $ z \in \cZ $, there exists a $\phi \in \Sigma(\Phi)$ such that $\fG_z = \fG_\phi$. It follows from Remark \ref{rem:parabolic} that if $\phi$ exists, it is necessarily unique.

	Hence, if $\cZ$ is $\fG$-compatible, we can associate it with a \emph{unique} $\fG$-stabilizer-preserving map $\lambda: \cZ \to \Sigma(\Phi)$ that maps $z \in \cZ$ to the unique $\phi \in \Sigma(\Phi)$ such that $\fG_z = \fG_\phi.$ 
	Moreover, for each $\phi \in \Sigma(\Phi)$, we define $\cZ(\phi)$ to be the subposet of $\cZ$ induced on $\lambda^{-1}(\phi) = \{ \z \in \cZ \ : \ \fG_z = \fG_\phi\}.$

\end{defn}

\begin{ex}\label{ex:set_partitions}
	Consider the poset $ \cS_{d} $ of set partitions of $ [d] $.
	This is similar to ordered set partitions $ \cO_{d} $ (Definition \ref{defn:ordered_set_partitions})  but blocks are unordered.
	The poset $ \cS_{d} $ is not the face poset of any polytope when $ d > 2 $.
	Indeed, every interval of length 2 in a face poset must have four elements (this is called the \textit{Diamond property} \cite[Theorem 2.7 (iii)]{zie}), and $ \cS_{d} $ contains intervals of length 2 with more than four elements whenever $ d > 2 $.

	The poset $ \cS_{d} $ is $ \fS_{d} $-symmetric poset, but it is not $ \fS_{d} $-compatible for any choice of genertaing poset when $ d > 2 $. 
	Consider the poset $ \cS_{4} $ of set partitions of $ [4] $.
	Any generating subposet must contain an element of the form $ ( \left\{ a,b \right\}, \left\{ c,d \right\} ) $.
	The stabilizer for the action of $ \fS_{4} $ of such an element is the subgroup generated by the transpositions $ (ab), (cd) $ and the elements $ (ac)(bd), (ad)(bc) $.
	This subgroup does not stabilize any $ \phi \in \Phi_{3} $ (it is not a parabolic subrgroup), and thus no generating subposet of $ \cS_{4} $ is $\fG$-compatible.
\end{ex}

Recall the carrier map $ \carr $ for $\Phi$ is defined in Definition \ref{defn:carr_map}.
\begin{ex}
  Let $ P $ be a $\Phi$-placed and $\Psi$-appropriate polytope. 
	Let $\cT(P)$ be the face poset of $\Sigma(\fG(P))$ and $\fG$-action is the one we have been using. 
	Then $\cT(P)$ is $\fG$-symmetric with the generating subposet $\cZ(P)$, the face poset of $\ffan(P)$. Furthermore, $ \cZ(P) $ is $\fG$-compatible with the carrier map $\carr$ being its associated $\fG$-stabilizer-preserving map, and for each $\phi \in \Sigma(\Phi)$, we have $\cZ(\phi) = \cZ(P, \phi)$. 

\end{ex}

\begin{thm}\label{thm:sym-poset}
  Suppose $\cT$ is $\fG$-symmetric with a $\fG$-compatible generating subposet $\cZ$, and suppose $\lambda: \cZ \to \Sigma(\Phi)$ is the $\fG$-stabilizer-preserving map associated with $\cZ$. 
	Let $\cR$ be the $\Sigma(\Phi)$-poset $\cZ^\lambda$. 
	Then the $\fG$-symmetrization $\fG\cR$ of $\cR$ is isomorphic to $\cT.$
\end{thm}

\begin{proof}
  The proof is basically the same as that of Theorem \ref{thm:symposet} by defining a map $f: gz \mapsto (z, g\lambda(z))$, and proving that it is a poset isomorphism from $\cT$ to $\fG\cR$.
\end{proof}

Now we are ready to lay a general framework to realize symmetric posets as polytopes.
	The most prominent examples of this kind of realization problems are connected: Reiner-Ziegler Coxeter permuto-associahedra \cite{reiner1994coxeter}, Barali\'c-Ivanovi\'c-Petri\'c simple permuto-associahedron \cite{baralic_ivanovic_petric}, and Gaiffi's permutonestohedra \cite{gaiffi}.
	Our results provide a unified way to approach these problems.

	\subsection{A recipe for realizing a symmetric poset} \label{subsec:recipe}

Given a finite $\fG$-symmetric poset $\cF,$ Proposition \ref{prop:ZPPhi}, Theorem \ref{thm:symposet} and Theorem \ref{thm:sym-poset} provide a recipe for realizing $\cF$, summarized in the following three steps: 
\begin{enumerate}[label={\textbf{\Roman*}.}, leftmargin=10mm, itemsep=2mm]
  \item \textbf{Finding the generating set $\cZ$ and its associated map $\lambda$:} Let $\cT$ be the dual of $\cF$. Find a generating subposet $ \mathcal{Z} $ of $ \mathcal{T} $ that is $\fG$-compatible, and let $\lambda: \cZ \to \Sigma(\Phi)$ be the $\fG$-stabilizer-preserving map associated with $\cZ$. 
  Finally, compute $\cZ(\phi)$ for each $\phi \in \Sigma(\Phi)$ or at least for $\phi = \Phi$.
  \item \textbf{Realizing $\cZ^\lambda$ as the refined fundamental fan of a polytope:}
    \begin{enumerate}[label=(\alph*)] 
      \item\label{step:a} Realize the dual of $\cZ(\Phi)$ as a $\Phi$-placed and $\Psi$-appropriate polytope $ P $, and described the normal fan of $P$, and describe the $\Sigma(\Phi)$-poset $\cZ^\lambda$.

      \item\label{step:b} Verify that $\cR(P)$, the face poset of the refined fundamental fan of $P$, is isomorphic to $\cZ^\lambda$ as $\Sigma(\Phi)$-posets.

      \end{enumerate}
 
  \item \textbf{Symmetrizing the polytope:} Symmetrize $P$ to obtain its $\fG$-symmetrization $\fG(P)$, which is a desired realization of $\cF$.
\end{enumerate}

In general, Steps I and III in the recipe are straightforward. 
Verifying whether the dual of a given poset \(\mathcal{F}\) that is \(\mathfrak{G}\)-symmetric possesses a \(\mathfrak{G}\)-compatible generating subposet may require effort. However, if \(\mathcal{F}\) is finite, this verification is feasible.
Therefore, the approach laid out in the recipe reduces the problem of realizing a $\fG$-symmetric poset $\cF$ to the problem of realizing a $\fG$-generator of its dual described in Step II.

\subsubsection*{Alternative steps for Step II}

We note that Step \ref{step:a} is the key and most difficult part of Step II, as we have to construct a \emph{correct} polytope $P$ so that Step (b) will be successful. 
Therefore, one might want to split Step \ref{step:a} into two steps: 
\begin{enumerate}[label=(a.\roman*)]
  \item\label{step:ai} Realize the dual of $\cZ(\Phi)$ as a polytope $Q$ without worrying about whether it is $\Phi$-placed or $\Psi$-appropriate, and describe its normal fan.
  \item\label{step:aii} Find a suitable (invertible) affine transformation to convert $Q$ to a $\Phi$-placed and $\Psi$-appropriate polytope $P$, and compute its normal fan. 
\end{enumerate}
Finally, for Step \ref{step:b} of Step II, one can use various results we developed in this paper to complete this step. For example, it follows from Lemma \ref{lem:OmegaPiso} that Step (b) can be replaced with the following: 
\begin{enumerate}[label=(b')]%[label=(b.\roman*)]
  \item\label{step:bi} Verify that $\Omega(P)= \{ (\sigma, \phi) \in \Sigma(P) \times \Sigma(\Phi) \ : \ \sigma^\circ \cap \phi^\circ \neq \emptyset \}$ ordered by \eqref{eq:OmegaPorder} is isomorphic to $\cZ^\lambda$ as $\Sigma(\Phi)$-posets.
\end{enumerate}
One benefit of using the alternative step \ref{step:bi} is that we can use the following lemma to turn the problem of verifying $\Sigma(\Phi)$-poset isomorphism to an equivalent poset isomorphism problem. We say a map $f: \cZ \to \Omega(P)$ is \emph{$\lambda$-consistent}, if for every $z \in \cZ$, we have $f(z) = (\sigma, \lambda(z))$ for some $\sigma$. 
\begin{lem}\label{lem:convertiso}
  Suppose $\cZ$, $\lambda$, and $P$ are produced as in Step I and Step II/\ref{step:a} of the recipe. 
  Then there exists a poset isomorphism $f$ from $\cZ$ to $\Omega(P)$ that is $\lambda$-consistent if and only if there exists a $\Sigma(\Phi)$-poset isomorphism $g$ from $\cZ^\lambda$ to $\Omega(P)$. 
\end{lem}

\begin{proof} Let $h$ be the poset isomorphism from $\cZ$ to $\cZ^\lambda$ defined by mapping $z$ to $(z, \lambda(z))$. One checks that if $f$ is a poset isomorphism from $\cZ$ to $\Omega(P)$ that is $\lambda$-consistent, then $g :=f \circ h$ is a desired $\Sigma(\Phi)$-poset isomorphism $g$. Conversely, if $g$ is a $\Sigma(\Phi)$-poset isomorphism from $\cZ^\lambda$ to $\Omega(P)$, then $f := g \circ h^{-1}$ is a desired $\lambda$-consistent poset isomorphism. 
\end{proof}

\subsubsection*{Necessary conditions}

For the general problem of realizing a poset $\cF$ as a polytope, one often starts by checking whether $\cF$ satisfies some simple necessary conditions for it to be polytopal, e.g, whether $\cF$ has the Diamond property. 
Similarly, after one completes Step I and before moves on to Step II of our recipe, it is a good idea to verify a few known necessary conditions for $\cZ^\lambda$ to be realized as the refined fundamental fan $\cR(P)$ of a poltyope $P$. We list a few necessary conditions below, as consequences of Lemma \ref{lem:orderconsist}, Corollary \ref{cor:nonempty}, and Lemma \ref{lem:maximalelts}, respectively:
\begin{enumerate}[label=(N\arabic*)]
  \item\label{itm:N1} $\lambda$ is order-preserving, or equivalently, $\cZ^\lambda$ is order-consistent. 
  \item\label{itm:N2} $ \lambda $ is surjective, or equivalently, $\cZ(\phi)$ is nonempty for each $\phi \in \Sigma(\Phi)$.
  \item\label{itm:N3} The maximal elements of $\cZ(\phi)$ are of rank $\dim(\phi)$, for each $\phi \in \Sigma(\Phi)$.
  \end{enumerate}

  The necessary conditions listed above are all ones that can be easily verified before Step II. However, there are also necessary conditions that are harder to verify and are related to Step II: 
  \begin{enumerate}[label=(N\alph*)]
    \item\label{itm:Na} $\cZ(\Phi)$ is polytopal, i.e., it is isomorphic to the face poset of a polytope.
    \item\label{itm:Nb} There exists a fan $ \Gamma $ supported on the fundamental chamber $ \Phi $ such that $\cF(\Gamma)^\carr$ and $\cZ^\lambda$ are isomorphic $\Sigma(\Phi)$-posets. 
  \end{enumerate}
  One sees that the difficulty of checking condition \ref{itm:Na} is basically the same as that of completing Step \ref{step:ai}. 
	Although condition \ref{itm:Nb} seems to be related to Step \ref{step:b}, it really is a condition one should think of during Step \ref{step:a} or Step \ref{step:aii}, as we said above that we need to make sure that the polytope constructed in Step \ref{step:a} is a \emph{correct} one. 
	Furtheremore, since $\cZ^\lambda$ tells us all the combinatorics about the refined fundamental fan, constructing a $\Gamma$ such that $\cF(\Gamma)^\carr$ is isomorphic to $\cZ^\lambda$ as $\Sigma(\Phi)$-posets could be very enlightening for completing Step \ref{step:aii}. 
	Note that checking condition \ref{itm:Nb} could be completely independent from checking \ref{itm:Na}, as $\cF(\Gamma)^\carr$ does not have to be equal to $\cZ^\lambda$. 
	Therefore, we ask the following question:

\begin{ques}
  Are the necessary conditions \ref{itm:Na} and \ref{itm:Nb} sufficient to conclude that $\cZ^\lambda$ can be realized as a refined fundamental fan of a polytope?
	
\end{ques}

\subsection{Extended example}
In this part, we will use an extended example to illustrate how to use the recipe we presented in \S \ref{subsec:recipe} to realize a $\fG_d$-symmetric poset as a $\fG_d$-symmetrization of a polytope.

The $\fG_d$-symmetric poset that we will realize is built using elements from the following two posets and can be seen as the hybrid of these two posets at end of this section:

\begin{enumerate}
  \item The poset $ \mathcal{O}_{d} $ of ordered set partitions of $ [d] $ in Definition \ref{defn:ordered_set_partitions}.
  \item The truncated Boolean poset $\widehat{\cB}_{d-1}$ of nonempty subsets of $[d-1]$ in Definition \ref{def:Boolean}.  
\end{enumerate}
We start by describing the elements in the poset. 
\begin{defn}
Let $ d	$ be a positive integer.
A \emph{decorated} ordered set partition of $ [d] $ is pair $ \cD = ( \mathcal{S}, B) $, where $ \mathcal{S} = ( S_{1}, \dots, S_{k} )  \in \cO_d$ is an ordered partition of $[d]$ with $k$ parts and $ B $ is a nonempty subset of $\Type(\cS)$ (recall Definition \ref{defn:ordered_set_partitions}) if $k > 1$ and $B=[d-1]$ if $k=1$.
In particular, this means that there is a unique decorated ordered set partition with $ \cS = ([d]) $, namely $ ([d],[d-1]) $, and it behaves differently from the rest.  We call it the \emph{trivial} decorated ordered set partition.

For each decorated ordered set partition $\cD$ of $[d]$, we represent it with a \emph{bar representation} defined as follows: 
If $\cD = (\cS, B)$ is a nontrivial decorated ordered set partition, we represent it by putting a bar $ \mid $ between $ S_i $ and $ S_{i+1} $ if $t_i \in B$ and a double bar $\mm$ between $S_i$ and $S_{i+1}$ otherwise. If $\cD = ([d], [d-1])$, we simply put all the elements together with no bars at all.

\end{defn}

\begin{ex}
  Let $\cS = (245, 17, 3, 68) \in \cO_8$ and $B =\{3,6 \} \subseteq \Type(\cS)=\{ 3, 5, 6\}$. 
	Then the pair $\cD=(\cS, B)$ is a decorated ordered set partition of $[8]$ with its bar representation $245 \mid 17 \mm 3 \mid 68$.

	The trivial decorated ordered set partition of $[8]$ is $([8], [7])$ and its bar representation is $12345678$.
\end{ex}

\begin{defn}
	\label{def:hybrid}
	Let $ \mathcal{F}_{d} $ be the poset on the set of decorated ordered set partitions of $[d]$ where the order is defined by
	\begin{equation}\label{eq:Forder}
	\mathcal{D} = (\cS, B) \preceq \mathcal{D}'=(\cS', B') \quad \text{if and only if} \quad  \text{$\cS \preceq \cS'$ in $ \cO_{d} $ and $ B \preceq B' $ in $ \widehat{\cB}_{d-1}$ } .
      \end{equation}
\end{defn}

\begin{rem}
  Note that the partial order on $\cF_d$ can be equivalently defined using bar representations. 
	It is easier to describe the covering relations:  $ \mathcal{D} \precdot \mathcal{D}' $ in $\cF_d$ if the bar representation of $ \mathcal{D}' $ can be obtained from that of $ \mathcal{D} $ by one of the following operations:
\begin{enumerate}[label={\bf b}\roman*.]
		\item merging two consecutive blocks in the bar representation of $ \mathcal{D} $ that are separated by a double bar.
		\item replacing a double bar in the bar representation of $ \mathcal{D} $ by a single bar.
		\item merging all blocks together if there are no double bars but at least one single bar in the bar representation of $\mathcal{D}$.
	\end{enumerate}

\end{rem}

\begin{ex}
Here is a maximal chain of $\cF_5$ in its bar representation:
\begin{small}
\[
	5\mm 2 \mid 4 \mm 1 \mm 3 \ \precdot \ 5 \mm 2 \mid 4 \mm 13 \ \precdot \ 5 \mid 2 \mid 4 \mm 13 \ \precdot \ 5 \mid 2 \mid 4 \mid 13 \ \precdot \ 12345.
\]
\end{small}
\end{ex}
It is straightforward to verify the following result from definitions:
\begin{lem}
  The poset $ \mathcal{F}_{d} $ is graded of rank $d-1$ where 
  the corank of a non-maximal element $\cD=(\cS, B)$ is $d - |B|$, which is the same as $1$ plus the number of double bars in its bar representation. 
\end{lem}

Each permutation $ \pi \in \fS_d $ acts on each decorated ordered set partition of $[d]$ naturally. Thus, $\cF_d$ is a $\fS_d$-symmetric poset. Below, we use the 3-step recipe we laid out in \S \ref{subsec:recipe} to solve the following realization problem:

\begin{prob}\label{prob:realization}
  Determine whether we can construct a polytope $P$ such that $\cF_d$ is the face poset of its symmetrization $\fG(P)$.
\end{prob}

Before we start, recall that Lemma \ref{lem:faces_Phi} says that the map $\cS \mapsto \sigma_\cS$ gives a poset isomorphism from $\cO_d^\std$ to the dual of the poset $\cF(\Sigma(\Phi_{d-1}))$, where $\sigma_\cS$ is defined in \eqref{eq:osp_cone_alt}. 
In our process of realizing $\cF_d$ below, we will always use $\sigma_\cS$ to describe cones in $\Sigma(\Phi_{d-1})$.

\subsection*{Step I: Finding generating set and its associated map}

Let $\cT_d$ be the dual poset of $\cF_d$ and let $\cZ_d$ be its induced subposet on all the decorated \emph{standard} ordered set partitions, i.e, all the $\cD=(\cS, B)$ with $\cS \in \cO_d^{\std}$. 
Then $ \mathcal{T}_{d} $ is $ \fS_{d} $-symmetric and $ \mathcal{Z}_{d} $ is a generating subposet that is $\fS_d$-compatible. More precisely, for each $\cD=(\cS, B) \in \cZ_d,$ we have
\begin{equation}\label{eq:stabilizers_example}
(\fS_d)_\cD = (\fS_d)_\cS = (\fS_d)_{\sigma_\cS}.
\end{equation} 
Hence, the $\fS_d$-stabilizer-preserving map associated with $\cZ_d$ is defined as $\lambda: \cZ_{d} \to \Sigma( \Phi_{d-1} )$ by $ \lambda(\cS, B) = \sigma_\cS$. 
Moreover, for each $([d])\neq\cS \in \cO_{d}^\std,$ we have that $\cZ_d(\sigma_\cS)$ is the subposet of $\cZ_d$ induced on $\{ (\cS, B) \ : \ \emptyset \neq B \subseteq \Type(\cS)\}$, which is clearly isomorphic to the dual of the truncated Boolean poset $\widehat{\cB}_{|\cS|-1}$.

In Figure \ref{fig:z3_hybrid} we have depicted $ \mathcal{Z}_{4} $ in bar representations in which different $\cZ_4(\sigma_\cS)$'s are highlighted with different colors.

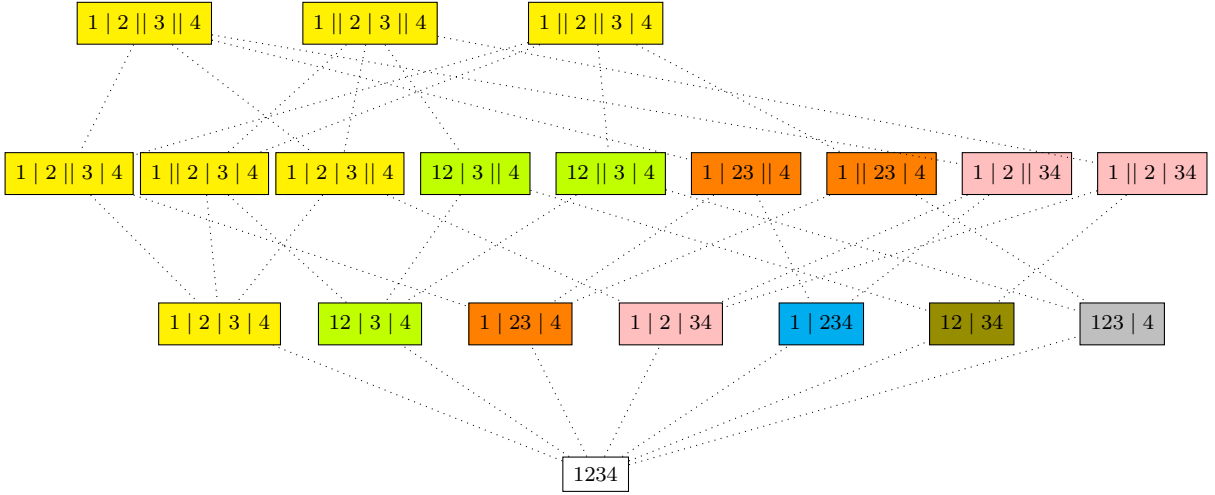
\begin{figure}
\begin{tikzpicture}
	\node at (1,4) [rectangle, draw, fill=yellow] (18){\tiny $ 1 \mid 2\mm 3\mm 4 $};
	\node at (4,4) [rectangle, draw, fill=yellow] (19){\tiny $ 1\mm 2 \mid 3\mm 4 $};
	\node at (7,4) [rectangle, draw, fill=yellow] (20){\tiny $ 1\mm 2\mm 3 \mid 4 $};
	
	\node at (0*1.8,2) [rectangle, draw, fill=yellow] (14){\tiny $ 1 \mid 2\mm 3 \mid 4 $};
	\node at (1*1.8,2) [rectangle, draw, fill=yellow] (15){\tiny $ 1\mm  2 \mid 3 \mid 4 $};
	\node at (2*1.8,2) [rectangle, draw, fill=yellow] (16){\tiny $ 1 \mid 2 \mid 3 \mm  4 $};
	\node at (3*1.8,2) [rectangle, draw, fill = lime] (11){\tiny $ 12 \mid 3\mm 4 $};
	\node at (4*1.8,2) [rectangle, draw, fill = lime] (13){\tiny $ 12\mm  3 \mid 4 $};
	\node at (5*1.8,2) [rectangle, draw, fill=orange] (9){\tiny $ 1 \mid 23\mm 4 $};
	\node at (6*1.8,2) [rectangle, draw, fill = orange] (12){\tiny $ 1\mm 23 \mid 4 $};
	\node at (7*1.8,2) [rectangle, draw, fill = pink] (8){\tiny $ 1 \mid 2\mm 34 $};
	\node at (8*1.8,2) [rectangle, draw, fill = pink] (10){\tiny $ 1 \mm  2 \mid 34 $};
	
	\node at (2,0) [rectangle, draw, fill=yellow] (3){\tiny $ 1 \mid 2 \mid 3 \mid 4 $};
	\node at (4,0) [rectangle, draw, fill = lime] (4){\tiny $ 12 \mid 3 \mid 4 $};
	\node at (6,0) [rectangle, draw, fill = orange] (5){\tiny $ 1 \mid 23 \mid 4 $};
	\node at (8,0) [rectangle, draw, fill = pink] (6){\tiny $ 1 \mid 2 \mid 34 $};
	\node at (10,0) [rectangle, draw, fill=cyan] (1){\tiny $ 1 \mid 234 $};
	\node at (12,0) [rectangle, draw, fill = olive] (2){\tiny $ 12 \mid 34 $};
  \node at (14,0) [rectangle, draw, fill = lightgray] (0){\tiny $ 123 \mid 4 $};

  \node at (7,-2) [rectangle, draw] (00){\tiny $ 1234 $};

	\draw[dotted] (18)--(8) (18)--(9) (18)--(14) (18)--(16) (19)--(10) (19)--(11) (19)--(15) (19)--(16) (20)--(12) (20)--(13) (20)--(15) (20)--(14);
	\draw[dotted] (8)--(1) (8)--(6) (9)--(5) (9)--(1) (10)--(6) (10)--(2) (11)--(2) (11)--(4) (12)--(5) (12)--(0) (13)--(4) (13)--(0) (14)--(3) (14)--(5) (15)--(3) (15)--(4) (16)--(3) (16)--(6) ;
	\draw[dotted] (00)--(0) (00)--(1) (00)--(2) (00)--(3) (00)--(4) (00)--(5) (00)--(6);
	\end{tikzpicture}
	\caption{Hasse diagram for $ \mathcal{Z}_{4} $ a generating set for $ \mathcal{T}_{4} $. Different colors highlight $\cZ_4(\sigma_\cS)$ for different $\cS$'s.}
\label{fig:z3_hybrid}
\end{figure}

Following Definition \ref{defn:Sigma-poset},  we have that $\cZ_d^\lambda$ is the $\Sigma(\Phi_{d-1})$-poset on the set
\[ \{ ( (\cS, B), \sigma_\cS) \ : \ (\cS, B) \in \cZ_d\},\]
ordered by
\[ ((\cS, B), \sigma_\cS) \le ((\cS', B'), \sigma_{\cS'}) \text{ if and only if } (\cS, B) \le (\cS', B') \text{ in $\cZ_d$}.\]

\subsection*{Step II: Realizing \texorpdfstring{$\cZ_d^\lambda$ as the refined fundamental fan of a polytope}{refined fundamental fan}}
\ 
In this part, instead of using the original steps \ref{step:a} and \ref{step:b} described in the recipe, we will use the alternative steps \ref{step:ai}, \ref{step:aii}, and \ref{step:bi}.

\subsubsection*{Step \ref{step:ai}} 
Note that $\Phi_{d-1} = \sigma_{\cS_0},$ where $\cS_0 = (1, 2, \dots, d) \in \cO_d^\std$. Hence, the dual of $\cZ_d(\Phi_{d-1})$ is isomorphic to the truncated Boolean poset $\widehat{\cB}_{d-1},$ which is the face poset of a $(d-2)$-simplex. 
The most natural realization of $\widehat{\cB}_{d-1}$ is the standard $(d-2)$-simplex $ \Delta_{d-2} = \conv \left\{ \mathbf{e}_{i} ~:~ 1\leq i \leq d-1 \right\} \subseteq \mathbb{R}^{d-1}$. Following notations in Example \ref{ex:simplex}, we can describe all of its normal cones: %for each nonempty $B \subseteq [d-1]$, 
for each $B \in \widehat{\cB}_{d-1}$,
\begin{equation}\label{eq:ncone_simplex_shift} \ncone(\Delta_B, \Delta_{d-2}) = \sigma_{d-1,B} = \left\{ \y \in \R^{d-1}  \ : \ 
    \begin{array}{c} y_i = y_j \text{ if $i,j \in B$} \\
y_i \le y_j \text{ if $i \notin B$ and $j \in B$} \end{array} \right\}.\end{equation}

\subsubsection*{Step \ref{step:aii}}
\ 
Let $ \bgamma \in \mathbb{R}^{d} $ be any vector. 
Recall from Example \ref{defn:typeA} that $U_{d-1}^\bgamma$ is a translation of the vector space $V_{d-1}$, which has a basis $\{\a_i = -\e_i+\e_{i+1} \ : \ 1 \le i \le d-1\}$.
We define an invertible affine transformation from $\R^{d-1}$ to $U_{d-1}^\bgamma \subseteq \mathbb{R}^{d}$ as follows: % for each $\x =(x_1, \dots, x_{d-1}) = \sum_{i=1}^{d-1} x_i \ee_i,$
for each vector $ \x =(x_{1}, \dots, x_{d-1}) =  \sum_{i=1}^{d-1} x_i \e_i \in \mathbb{R}^{d-1} $ we define
\begin{equation}
\label{eqn:phi_alpha}
T_\bgamma(\x) = T_\bgamma\left( \sum_{i=1}^{d-1} x_i \e_i \right)\defeq \bgamma + \sum_{i=1}^{d-1}  x_i \a_i. 
\end{equation}

We now use a particular $\bgamma$ to construct our polytope $P$ (that will be proved to be the polytope we are seeking for). 
\begin{construct}
  Let  $\bgamma_0 = (1^{2},2^{2},\dots,d^{2}) \in \mathbb{R}^{{d}} $ and $ P = T_{\bgamma_0}(\Delta_{d-2}) $. 
\end{construct}

\begin{lem}
 
  The polytope $P= T_{\bgamma_0}(\Delta_{d-2})$ is $ \Phi_{d-1} $-placed and $ \Psi_{d-1} $-appropriate. 
\end{lem}
\begin{proof}
  The vertices of $P$ are $\{ \bgamma + \a_i \ : \ 1 \le i \le d-1\}$. It is easy to verify that all vertices have strictly increasing entries, and hence are in $\Psi_{d-1}^\circ$. Therefore, $P \subseteq \Psi_{d-1}^\circ$, and thus is $ \Psi_{d-1} $-appropriate. Let $ \balpha = (1,2,\dots,d) $.
  As the functional $ \langle \balpha, \cdot \rangle $ is constant on all the vertices of $P$, we conclude it is constant on $P.$ Since $ \balpha \in \Phi_{d-1}^\circ $, we have that $ P $ is $ \Phi_{d-1} $-placed.
\end{proof}

\begin{rem}
The particular choice of $ \bgamma_0  $ is not too important but notice that taking $ \bgamma_0 = \balpha $ does not work since the resulting vertices are not strictly increasing.
\end{rem}

We can describe the normal cones of $ T_{\bgamma}(\Delta_{d-2}) $ (for any $\bgamma$) via the following basic lemma. 
Recall again from Example \ref{defn:typeA} that $W_{d-1}$ is the dual space of $V_{d-1}$ and has a corresponding dual basis $\left\{\f_{1}^{(d)}, \dots, \f_{d-1}^{(d)} \right\}$ defined as in \eqref{eq:defnfi}.
Also, for convenience, for each $\w \in W_{d-1}$, we define its \emph{first difference vector} of $\w$ to be
\begin{equation}
	\bdelta(\w) = (w_2-w_1, w_3-w_2, \dots, w_d-w_{d-1}) \in \mathbb{R}^{d-1},
\end{equation}
and for each $1 \le i \le d-1,$ we denote its $i$th entry by $\delta_i(\w) := w_{i+1} - w_{i}.$

\begin{lem}\label{lem:ncone_transformation} 
	Let $ Q \subseteq \mathbb{R}^{d-1} $ be a polytope and its normal fan is defined in $\R^{d-1}$, which is identified with its own dual via the inner product $ \left[\cdot, \cdot\right] $.
Then for any face $ F \subseteq Q  $ we have that 
\begin{align}
  \ncone(T_{\bgamma}(F),T_{\bgamma}(Q)) =& \left\{ \sum_{i=1}^{d-1} y_i \f_i^{(d)} \in W_{d-1} ~:~ \y \in \ncone(F,Q) \right\} \label{eq:ncone_1st} \\
  =& \left\{ \w \in W_{d-1} ~:~ \bdelta(\w) \in \ncone(F,Q) \right\}. \label{eq:ncone_2nd}
\end{align}
\end{lem}

\begin{proof}
  For any $\x, \x', \y \in \R^{d-1},$ one checks that $[ \x - \x', \y ] = \left\langle T_\bgamma(\x) - T_\bgamma(\x'), \sum_{i=1}^{d-1} y_i \f_i^{(d)}\right\rangle.$
  Hence, we have $ [\x, \y] \le [\x', \y]$ if and only if $\left\langle T_\bgamma(\x), \sum_{i=1}^{d-1} y_i \f_i^{(d)}  \right\rangle \le \left\langle T_\bgamma(\x'), \sum_{i=1}^{d-1} y_i \f_i^{(d)}  \right\rangle.$
This, together with the definition of normal cones, implies Equality \eqref{eq:ncone_1st}.

Letting $\sum_{i=1}^{d-1} y_i \f_i^{(d)} = \w = \sum_{i=1}^d w_i \e_i,$ we obtain the relation that 
$y_i = w_{i+1}-w_i = \delta_i(\w),$ for each $1 \le i \le d-1$. Hence, Equality \eqref{eq:ncone_2nd} follows. 
\end{proof}

We now describe the normal fan of  $ T_{\bgamma}(\Delta_{d-2}) $.
\begin{lem}\label{lem:nu_cone}
  For each $ B \in \widehat{\cB}_{d-1} $, the normal cone $\ncone(T_\bgamma(\Delta_B), T_\bgamma(\Delta_{d-2}))$ of $ T_\bgamma(\Delta_{d-2}) $ at the face $T_\bgamma(\Delta_B)$ is given by
\begin{small}
\begin{equation}\label{eq:ncone_P}
  \nu_B \defeq\left\{ \w \in W_{d-1}  \ : \ 
    \begin{array}{c} \delta_i(\w) = \delta_j(\w) \text{ if $i,j \in B$} \\
\delta_i(\w) \le \delta_j(\w) \text{ if $i \notin B$ and $j \in B$} \end{array} \right\}.
\end{equation}
\end{small}
Moreover, the map $B \mapsto \nu_B$ is a poset isomorphism from $\widehat{\cB}_{d-1}$ to the dual of the face poset of $\Sigma(T_\bgamma(\Delta_{d-2}))$.
\end{lem}

\begin{proof}
  Applying Lemma \ref{lem:ncone_transformation} to Equation \eqref{eq:ncone_simplex_shift}, we immediately obtain a description for the normal cones of $ T_{\bgamma}(\Delta_{d-2}) $ as shown in \eqref{eq:ncone_P}. 

  The second conclusion the follows Lemma \ref{lem:face&fan}, Example \ref{ex:simplex}, and the fact that $T_\bgamma$ is an invertible affine transformation. 
\end{proof}

\subsubsection*{Step \ref{step:bi}}
Recall that $\Omega(P)$ is the poset on $\{ (\sigma, \phi) \in \Sigma(P) \times \Sigma(\Phi) \ : \ \sigma^\circ \cap \phi^\circ \neq \emptyset \}$ with ordering relation
\begin{equation}\label{eq:OmegaPorder-copy} (\sigma, \phi) \le (\sigma', \phi') \text{ in $\Omega(P)$ \quad if and only if \quad} \sigma \subseteq \sigma' \text{ and } \phi \subseteq \phi'. %\sigma \cap \phi \subseteq \sigma' \cap \phi'.
\end{equation}

We first figure out the set $\Omega(P)$ when $P = T_{\bgamma_0}(\Delta_{d-2})$ by giving a characterization for when $\sigma_\cS^\circ \cap \nu_B^\circ \neq \emptyset$.

\begin{prop}\label{prop:charintersect}
  For $\cS \in \cO_d^\std$ and $B \in \widehat{\cB}_{d-1},$ we have $\sigma_\cS^\circ \cap \nu_B^\circ \neq \emptyset$ if and only if $(\cS, B) \in \cZ_d$, that is, $(\cS, B)$ satisfies one of the followings conditions:
  \begin{enumerate}
    \item $\cS = ([d])$ and $B=[d-1].$
    \item $\cS \neq ([d])$ and $B$ is a (nonempty) subset of $\Type(\cS)$.
  \end{enumerate}
  As a consequence, for $ P = T_{\bgamma_0}(\Delta_{d-2}) $, we have
  \[ \Omega(P) = \{ (\nu_B, \sigma_\cS) \ : \ (\cS, B) \in \cZ_d\}.\]
\end{prop}

\begin{proof}[Proof of Proposition \ref{prop:charintersect}]
  The cones $ \sigma_{\cS} $ and $ \nu_{B} $ are described in and Equations \eqref{eq:osp_cone_alt} and \eqref{eq:ncone_P}, respectively.
  It can be verified that the interiors of these cones are obtained by replacing $\le$ with $<$ %{using strict inequalities} 
  in these descriptions. One sees that $\sigma_\cS^\circ$ and $\nu_B^\circ$ can both be described using first difference vectors: 
  \begin{align}
    \sigma_\cS^\circ =& \left\{ \w \in W_{d-1} \ : \begin{array}{c} \delta_i(\w) = 0 \text{ if $i \notin \Type(\cS)$} \\
    \delta_i(\w) > 0 \text{ if $i \in \Type(\cS)$} \end{array} \right\}, \label{eq:sigma_int} \\
 \nu_B^\circ =& \left\{ \w \in W_{d-1}  \ : \ 
    \begin{array}{c} \delta_i(\w) = \delta_j(\w) \text{ if $i,j \in B$} \\
  \delta_i(\w) < \delta_j(\w) \text{ if $i \notin B$ and $j \in B$} \end{array} \right\}. \label{eq:nu_int}
\end{align}
If $\cS=([d])$, then $\sigma_\cS^\circ = \0$, and thus $\sigma_\cS^\circ \cap \nu_B^\circ \neq \emptyset$ if and only if $0 \in \nu_B^\circ$, which only happens when $B = [d-1].$ 

Suppose $\cS \neq ([d]).$ Then $\Type(\cS)$ is nonempty. If $B$ is a nonempty subset of $\Type(\cS)$, one checks $\ds \sum_{i \in B} 2\f_i^{(d)} + \sum_{i \in \Type(\cS)\setminus B} \f_i$ is in $\sigma_\cS^\circ \cap \nu_B^\circ,$ and hence $\sigma_\cS^\circ \cap \nu_B^\circ \neq \emptyset$. Conversely, assume $\sigma_\cS^\circ \cap \nu_B^\circ \neq \emptyset$ and let $\w \in \sigma_\cS^\circ \cap \nu_B^\circ,$ we need to show that $B$ is a (nonempty) subset of $\Type(\cS)$. By the description \eqref{eq:sigma_int} for $\sigma_\cS^\circ,$ we have that 
\begin{equation}\label{eq:deltainq} \delta_{i}(\w) < \delta_j(\w) \text{ for all $i \not\in \Type(\cS)$ and $j \in \Type(\cS)$}.
\end{equation}
However, by the first condition in the description \eqref{eq:nu_int} and the fact that $B \in \widehat{\cB}_{d-1}$ is not empty, we must have that $B$ is either a subset of $\Type(\cS)$ or has no intersection with $\Type(\cS).$ Suppose $B \cap \Type(\cS) = \emptyset. $ We pick $a \in \Type(\cS)$ and $b \in B$ (this can be done because both sets are nonempty). Then $a \not\in B$ and $b \not\in \Type(\cS)$. Thus, using the second condition in the description \eqref{eq:nu_int}, we conclude $\delta_a(\w) < \delta_b(\w).$ However, this contradicts \eqref{eq:deltainq}. Therefore, we conclude that $B$ is a subset of $\Type(\cS)$, completing the proof.  
\end{proof}

\begin{cor}
  Let $ P = T_{\bgamma_0}(\Delta_{d-2}) $.
  The map $(\cS, B) \mapsto (\nu_B, \sigma_\cS)$ gives a poset isomorphism from $\cZ_d$ to $\Omega(P)$. 
  Hence, the map $( (\cS, B), \sigma_\cS) \mapsto (\nu_B, \sigma_\cS)$ gives a $\Sigma(\Phi_{d-1})$-poset isomorphism from $\cZ_d^\lambda$ to $\Omega(P)$.
\end{cor}

\begin{proof}
  It follows from Proposition \ref{prop:charintersect} that the map $(\cS, B) \mapsto (\nu_B, \sigma_\cS)$ is a bijection from $\cZ_d$ to $\Omega(P)$. 
  Next, note that $\cZ_d$ is a subposet of $\cT_d$ which is dual to $\cF_d.$ Therefore, for any $(\cS, B), (\cS', B') \in \cZ_d$, we have 
\[
  (\cS, B) \preceq (\cS', B') \text{ in $\cZ_d$} \quad \text{if and only if} \quad  \text{$\cS' \preceq \cS$ in $ \cO_{d} $ and $ B' \preceq B $ in $ \widehat{\cB}_{d-1}$ }.
\]
It then follows from Lemma \ref{lem:faces_Phi}, Lemma \ref{lem:nu_cone}, and \eqref{eq:OmegaPorder-copy} that
\[
(\cS, B) \preceq (\cS', B') \text{ in $\cZ_d$} \quad \text{if and only if} \quad  (\nu_B, \sigma_\cS) \le (\nu_{B'}, \sigma_{\cS'}) \text{ in $\Omega(P)$}.
\]
Therefore, the first conclusion follows. 

The second conclusion follows from the first conclusion and Lemma \ref{lem:convertiso}.
\end{proof}

\subsection*{Step III: Symmetrizing the polytope} 

Now we are in position to answer positively Problem \ref{prob:realization}:
The polytope $ \fS_{d}(P) = \conv \left\{ \pi(\bgamma_0)) + \pi(\mathbf{a}_{i}) ~:~ \pi \in \fS_{d}, i\in [d-1] \right\}$ has face poset isomorphic to $ \cF_{d} $.

\bibliography{biblio}

\begin{thebibliography}{10}

\bibitem{ardila2020bipermutahedron}
Federico Ardila.
\newblock The bipermutahedron.
\newblock {\em arXiv preprint arXiv:2008.02295}, 2020.

\bibitem{baralic_ivanovic_petric}
Djordje Barali\'{c}, Jelena Ivanovi\'{c}, and Zoran Petri\'{c}.
\newblock A simple permutoassociahedron.
\newblock {\em Discrete Math.}, 342(12):111591, 18, 2019.

\bibitem{baumeister2007permutation}
Barbara Baumeister, Christian Haase, Benjamin Nill, and Andreas Paffenholz.
\newblock On permutation polytopes.
\newblock {\em Adv. Math.}, 222(2):431--452, 2009.

\bibitem{computing_symmetry}
David Bremner, Mathieu Dutour~Sikiri\'{c}, Dmitrii~V. Pasechnik, Thomas Rehn,
  and Achill Sch\"{u}rmann.
\newblock Computing symmetry groups of polyhedra.
\newblock {\em LMS J. Comput. Math.}, 17(1):565--581, 2014.

\bibitem{castillo2023lineup}
Federico Castillo and Jean-Philippe Labbé.
\newblock Lineup polytopes of product of simplices, 2023.

\bibitem{castillo2021effective}
Federico Castillo, Jean-Philippe Labbé, Julia Liebert, Arnau Padrol, Eva
  Philippe, and Christian Schilling.
\newblock An effective solution to convex $1$-body $n$-representability, 2021.

\bibitem{castillo2023permuto}
Federico Castillo and Fu~Liu.
\newblock The permuto-associahedron revisited.
\newblock {\em European Journal of Combinatorics}, 110:103706, 2023.

\bibitem{ewald}
G{\"u}nter Ewald.
\newblock {\em Combinatorial convexity and algebraic geometry}, volume 168 of
  {\em Graduate Texts in Mathematics}.
\newblock Springer-Verlag, New York, 1996.

\bibitem{gaiffi}
Giovanni Gaiffi.
\newblock Permutonestohedra.
\newblock {\em J. Algebraic Combin.}, 41(1):125--155, 2015.

\bibitem{hohlweg2011permutahedra}
Christophe Hohlweg, Carsten~EMC Lange, and Hugh Thomas.
\newblock Permutahedra and generalized associahedra.
\newblock {\em Advances in Mathematics}, 226(1):608--640, 2011.

\bibitem{humphreys}
James~E. Humphreys.
\newblock {\em Reflection groups and {C}oxeter groups}, volume~29 of {\em
  Cambridge Studies in Advanced Mathematics}.
\newblock Cambridge University Press, Cambridge, 1990.

\bibitem{kapranov}
Mikhail~M. Kapranov.
\newblock The permutoassociahedron, {M}ac {L}ane's coherence theorem and
  asymptotic zones for the {KZ} equation.
\newblock {\em J. Pure Appl. Algebra}, 85(2):119--142, 1993.

\bibitem{novik2019tale}
Isabella Novik.
\newblock A tale of centrally symmetric polytopes and spheres.
\newblock {\em Recent trends in algebraic combinatorics}, pages 305--331, 2019.

\bibitem{postnikov2006faces}
Alex Postnikov, Victor Reiner, and Lauren Williams.
\newblock Faces of generalized permutohedra.
\newblock {\em Doc. Math.}, 13:207--273, 2008.

\bibitem{postnikov2009permutohedra}
Alexander Postnikov.
\newblock Permutohedra, associahedra, and beyond.
\newblock {\em Int. Math. Res. Not. IMRN}, (6):1026--1106, 2009.

\bibitem{reiner1994coxeter}
Victor Reiner and G\"{u}nter~M. Ziegler.
\newblock Coxeter-associahedra.
\newblock {\em Mathematika}, 41(2):364--393, 1994.

\bibitem{zie}
G{\"u}nter~M. Ziegler.
\newblock {\em Lectures on polytopes}, volume 152 of {\em Graduate Texts in
  Mathematics}.
\newblock Springer-Verlag, New York, 1995.

\end{thebibliography}
\bibliographystyle{plain}

\end{document}